\pdfoutput=1
\RequirePackage{ifpdf}
\ifpdf 
\documentclass[pdftex]{sigma}
\else
\documentclass{sigma}
\fi

\usepackage{mathrsfs}
\usepackage{tikz}

\newcommand{\hyp}[5]{\,\mbox{}_{#1}F_{#2}\!\left(
 \genfrac{}{}{0pt}{}{#3}{#4};#5\right)}

\newtheorem{thm}[lemma]{Theorem}

\newtheorem{rem}[lemma]{Remark}
\newtheorem{lem}[lemma]{Lemma}
\newtheorem{prop}[lemma]{Proposition}

\newcommand{\expe}{{\mathrm e}}

\newcommand{\kk}{k} 

\newcommand{\R}{\mathbb{R}}
\newcommand{\C}{\mathbb{C}}
\newcommand{\Z}{\mathbb{Z}}
\newcommand{\N}{\mathbb{N}}

\renewcommand{\arraystretch}{1.2}
\DeclareMathOperator{\Ec}{Ec}
\DeclareMathOperator{\Es}{Es}
\DeclareMathOperator{\Fc}{Fc}
\DeclareMathOperator{\Fs}{Fs}
\DeclareMathOperator{\Gc}{Gc}
\DeclareMathOperator{\Hc}{Hc}
\DeclareMathOperator{\Gs}{Gs}
\DeclareMathOperator{\Hs}{Hs}
\renewcommand{\r}{\mathbf{r}}

\newcommand{\q}{\mathbf{q}}
\newcommand{\0}{\mathbf{0}}
\DeclareMathOperator{\sn}{sn}
\DeclareMathOperator{\cn}{cn}
\DeclareMathOperator{\dn}{dn}
\DeclareMathOperator{\ssc}{sc}
\DeclareMathOperator{\nd}{nd}
\DeclareMathOperator{\cs}{cs}
\DeclareMathOperator{\cd}{cd}
\DeclareMathOperator{\ns}{ns}

\DeclareMathOperator{\dc}{dc}
\DeclareMathOperator{\arcsn}{arcsn}
\DeclareMathOperator{\arcsc}{arcsc}
\DeclareMathOperator{\sech}{sech}
\DeclareMathOperator{\ds}{ds}
\DeclareMathOperator{\F}{{\rm F}}

\numberwithin{equation}{section}
\numberwithin{corollary}{section}
\numberwithin{remark}{section}
\numberwithin{theorem}{section}
\numberwithin{lemma}{section}

\begin{document}

\allowdisplaybreaks

\newcommand{\arXivNumber}{2202.08918}

\renewcommand{\PaperNumber}{041}

\FirstPageHeading

\ArticleName{Expansion for a Fundamental Solution of Laplace's\\ Equation in Flat-Ring Cyclide Coordinates}

\ShortArticleName{Expansion for a Fundamental Solution of Laplace's Equation}

\Author{Lijuan BI~$^{\rm a}$, Howard S.~COHL~$^{\rm b}$ and Hans VOLKMER~$^{\rm c}$}

\AuthorNameForHeading{L.~Bi, H.S.~Cohl, H.~Volkmer}

\Address{$^{\rm a)}$~Department of Mathematics,
The Ohio State University at Newark,
Newark, OH 43055, USA}
\EmailD{\href{mailto:bi.146@osu.edu}{bi.146@osu.edu}}
\URLaddressD{\url{https://newark.osu.edu/directory/bi-lijuan.html}}

\Address{$^{\rm b)}$~Applied and Computational
Mathematics Division, National Institute of Standards\\
\hphantom{$^{\rm b)}$}~and Technology, Mission Viejo, CA 92694, USA}
\EmailD{\href{mailto:howard.cohl@nist.gov}{howard.cohl@nist.gov}}
\URLaddressD{\url{http://www.nist.gov/itl/math/msg/howard-s-cohl.cfm}}

\Address{$^{\rm c)}$~Department of Mathematical Sciences, University of Wisconsin-Milwaukee,\\
\hphantom{$^{\rm c)}$}~Milwaukee, WI 53201-0413, USA}
\EmailD{\href{mailto:volkmer@uwm.edu}{volkmer@uwm.edu}}

\ArticleDates{Received November 20, 2021, in final form May 18, 2022; Published online June 03, 2022}

\Abstract{We derive an expansion for the fundamental solution of Laplace's equation in flat-ring coordinates in three-dimensional Euclidean space. This expansion is a double series of products of functions that are harmonic in the interior and exterior of ``flat rings''. These internal and external flat-ring harmonic functions are expressed in terms of simply-periodic Lam\'e functions. In a limiting case we obtain the expansion of the fundamental solution in toroidal coordinates.}


\Classification{35A08; 35J05; 33C05; 33C10; 33C15; 33C20; 33C45; 33C47; 33C55; 33C75}

\section{Introduction}
In 1875 Albert Wangerin \cite{Wangerin1875} introduced a coordinate system $s$, $t$ in the $xy$-plane by setting
$x+{\rm i}y=\dn(s+{\rm i}t,k)$, where $\dn$
\cite[formula~(22.2.6)]{NIST:DLMF} is one of the {Jacobian} elliptic functions. By rotating the system about the $y$-axis
he obtained coordinates $s$, $t$, $\phi$ in $\R^3$ that are called flat-ring
coordinates by Moon and Spencer \cite[p.~126]{MoonSpencer}. Actually, Moon and Spencer use $x+{\rm i}y=\sn(s+{\rm i}t,k)$ \cite[formula~(22.2.4)]{NIST:DLMF} but this leads to the same coordinate system.
Wangerin showed that the Laplace equation
\begin{equation}\label{1:Laplace}
\Delta u=\frac{\partial^2 u}{\partial x^2}+\frac{\partial^2 u}{\partial y^2}+\frac{\partial^2 u}{\partial z^2}=0
\end{equation}
can be solved by functions of the form
\begin{equation}\label{1:R}
u(x,y,z)={\mathcal R}(x,y,z) u_1(s) u_2(t) u_3(\phi),
\end{equation}
where ${\mathcal R}$ is a known elementary function called
the modulation factor (following Morse and~Fesh\-bach), and $u_1$, $u_2$, $u_3$ solve
ordinary differential equations.
In our case
${\mathcal R}=R^{-1/2}$, where~$R$ is the distance to the axis of rotation.
This is referred to as ${\mathcal R}$-separation of variables
and we express this by saying that
the Laplace equation is ${\mathcal R}$-separable in flat-ring coordinates.
Wangerin won a prize for his original work but only for the first part of his paper. In the second part
Wangerin tried to expand the fundamental solution of the Laplace equation (given as the inverse distance to a fixed point),
in terms of the separated solutions \eqref{1:R}. Such an expansion is needed in physical applications.
However, there were some obvious mistakes in this part of the paper. It~is the purpose of the present paper to
derive the desired expansion of the fundamental solution in flat-ring coordinates. It is a little bit surprising, but, as far as we know, this work has not been carried out in all those years since 1875.

Flat-ring coordinates are a special instance of so-called cyclidic coordinate systems. In these systems the coordinate surfaces
are of order 4 (they are zero sets of polynomials of degree $4$ in the Cartesian coordinates $x$, $y$, $z$). This is in contrast to quadric coordinate systems in which the coordinate surfaces have order 2 (like spherical coordinates).
B\^{o}cher \cite{Bocher} listed quadric and cyclidic coordinate systems in which Laplace's equation can be solved using separation of variables in his book ``\"{U}ber die Reihenentwicklungen der Potentialtheorie" published in 1894.
We~can also find these quadric and cyclidic coordinates systems in Miller's book \cite[Tables 14 on~p.~164 and~17 on~p.~210]{Miller}. The quadric coordinates are Cartesian, cylindrical, parabolic cylindrical, elliptic cylindrical, spherical,
prolate spheroidal, oblate spheroidal, parabolic, paraboloidal, sphero-conal, and ellipsoidal coordinates. The cyclidic coordinates are flat-ring cyclide, flat-disk cyclide, bi-cyclide, cap-cyclide, and 5-cyclide coordinates.
In the quadric systems we can choose $R=1$ in~\eqref{1:R} but this is not possible in the cyclidic systems.
In several of these coordinate systems the expansion of the fundamental solution is known; see \cite{Blimkeetal2008,CohlVolkcyclide2,Heine,MorseFeshbach}.

Let us consider {perhaps the simplest rotationally-invariant cyclidic coordinate system which ${\mathcal R}$-separates Laplace's equation in three-dimensions}, the toroidal coordinate system.
As we show in Section~\ref{toroidal} flat-ring coordinates reduce to toroidal coordinates in a limiting case.
Toroidal coordinates \cite[Section~8.10]{Lebedev} are given by
\begin{gather}\label{1:toroidal}
x=\frac{\sinh \tau \cos \phi}{\cosh \tau-\cos \psi},\qquad
y=\frac{\sinh \tau \sin \phi}{\cosh \tau-\cos \psi},\qquad
z=\frac{\sin \psi}{\cosh \tau-\cos \psi},
\end{gather}
where $\tau>0$, $-\pi<\phi,\psi\le \pi$. The coordinate surfaces $\tau={\rm constant}$ are tori, i.e.,
\[
\big(1+x^2+y^2+z^2\big)^2=4\big(x^2+y^2\big)\coth^2\tau.
\]
The coordinate surfaces $\psi={\rm constant}$ are spherical bowls, i.e.,
\[
(z-\cot \psi)^2+x^2+y^2=\frac{1}{\sin^2 \psi}.
\]
The coordinate surfaces $\phi={\rm constant}$ are half-planes, i.e.,
\[
x \sin \phi=y \cos \phi.
\]
Laplace's equation \eqref{1:Laplace} has solutions of the form
\[
u(x, y, z)=(\cosh \tau-\cos \psi)^{1/2} u_1(\tau)u_2(\psi)u_3(\phi),
\]
where $u_1$, $u_2$, $u_3$ satisfy the following ordinary differential equations, respectively:
\begin{gather}\label{1:ode1}
\frac{1}{\sinh \tau} \frac{\mathrm d}{{\mathrm d}\tau}\bigg(\sinh \tau \frac{{\mathrm d} u_1}{{\mathrm d}\tau}\bigg)-\bigg(n^2-\frac{1}{4}+\frac{m^2}{\sinh^2\tau}\bigg)u_1=0,
\\
\label{1:ode2}
\frac{{\mathrm d}^2u_2}{{\mathrm d}\psi^2}+n^2u_2=0,
\\
\label{1:ode3}
\frac{{\mathrm d}^2 u_3}{{\mathrm d}\phi^2}+m^2u_3=0,
\end{gather}
where $m$, $n$ are separation parameters. The general solution of \eqref{1:ode1} is
\[
u_1(\tau)=c_1P_{n-\frac{1}{2}}^m(\cosh\tau)+c_2Q_{n-\frac{1}{2}}^m(\cosh\tau),
\]
where $c_1$ and $c_2$ are constants, and $P_\nu^\mu\colon \C\setminus(-\infty,1]\to\C$ and $Q_\nu^\mu\colon \C\setminus(-\infty,1]{\to\C}$ are the associated Legendre functions of the first and second kind respectively
defined in terms of the Gauss hypergeometric function by \cite[functions~(14.3.6) and~(14.3.7)]{NIST:DLMF}\vspace{-1ex}
\begin{gather}
P_\nu^\mu(z):=\frac{1}{\Gamma(1-\mu)}
\bigg(\frac{z+1}{z-1}\bigg)^{\frac12\mu}\hyp21{-\nu,\nu+1}{1-\mu}{\frac{1-z}{2}}\!,
\label{Pdefn}
\end{gather}
and\vspace{-1ex}
\begin{gather}
Q_\nu^\mu(z):=
\frac{\expe^{{\rm i}\mu\pi}\sqrt{\pi} \, \Gamma(\nu+\mu+1)\big(z^2-1\big)^{\frac12\mu}}{2^{\nu+1} \Gamma\big(\nu+\frac32\big)z^{\nu+\mu+1}}
\hyp21{\frac{\nu+\mu+1}{2},\frac{\nu+\mu+2}{2}}{\nu+\frac32}{\frac{1}{z^2}}\!,
\end{gather}
where $\nu+\mu\not\in-\N$.

The internal toroidal harmonics \cite[equation~(8.11.8)]{Lebedev}
\[
G_{m,n}(x,y,z){:=}(\cosh \tau-\cos \psi)^{1/2} Q_{n-\frac{1}{2}}^m(\cosh\tau)\expe^{{\rm i}n\psi} \expe^{{\rm i}m\phi}\qquad\text{for}\quad m, n\in \Z
\]
are harmonic in $\R^3$ except for the $z$-axis. The external toroidal harmonics
\cite[equation~(8.11.9)]{Lebedev}
\[
H_{m,n}(x,y,z){:=}(\cosh \tau-\cos \psi)^{1/2}\, P_{n-\frac{1}{2}}^m(\cosh\tau)\expe^{{\rm i}n\psi} \expe^{{\rm i}m\phi}\qquad\text{for}\quad m, n\in \Z
\]
are harmonic in $\R^3$ except for the unit circle $z=0$, $x^2+y^2=1$.
The desired expansion of the fundamental solution of Laplace's equation in internal and external toroidal harmonics is
\begin{equation}\label{1:expansion}
\frac{1}{\|\r-\r^\ast\|}=\frac1\pi\sum_{m,n\in\Z} (-1)^m \frac{\Gamma\big(n-m+\frac12\big)}{\Gamma\big(n+m+\frac12\big)} G_{m,n}(\r)\overline{H_{m,n}(\r^\ast)}
\end{equation}
provided $\tau>\tau^\ast$, where $\tau$, $\tau^\ast$ are toroidal coordinates of~$\r=(x,y,z)$ and $\r^\ast=(x^\ast,y^\ast,z^\ast)$, respectively.
Formula \eqref{1:expansion} follows by inserting \cite[Theorem~8.795.2]{Grad}
{
\[
Q_\nu(\cosh\alpha)=\sum_{k=0}^\infty \epsilon_k(-1)^kP_\nu^{-k}(\cosh\tau)Q_\nu^k(\cosh\tau^\ast)\cos(k(\psi-\psi^\ast)),
\]
where $\epsilon_k:=2-\delta_{k,0}$ is the Neumann factor (commonly occurring in Fourier cosine series), $\cosh\alpha:=\cosh\tau\cosh\tau^\ast-\sinh\tau\sinh\tau^\ast\cos(\psi-\psi^\ast)$,
with $\nu=n-\frac12$,} into the single summation
expression for the reciprocal distance between two points in toroidal coordinates given by Bateman
\cite[Section~10.3, equation~(26)]{Bateman}
\begin{gather*}
\frac{1}{\|\r-\r^\ast\|}=
\frac{\sqrt{\cosh\tau-\cos\psi}\sqrt{\cosh\tau^\ast-\cos\psi^\ast}} {\sqrt{2}\,a\sqrt{\cosh\alpha-\cos(\psi-\psi^\ast)}}
\\ \hphantom{\frac{1}{\|\r-\r^\ast\|}}
{}=\frac{\sqrt{\cosh\tau-\cos\psi}\sqrt{\cosh\tau^\ast-\cos\psi^\ast}}{a\pi}
\sum_{n=-\infty}^\infty
Q_{n-\frac12}(\cosh\alpha)\expe^{{\rm i}n(\psi-\psi^\ast)},
\end{gather*}
which is a direct consequence of Heine's reciprocal square root identity\vspace{-1ex}
\cite[formula~(A5)]{CT}
\begin{equation*}
\frac{1}{\sqrt{z-x}}=
\frac{\sqrt{2}}{\pi}
\sum_{n=-\infty}^\infty\expe^{{\rm i}n\theta}Q_{n-\frac12}(z),
\end{equation*}
where $x=\cos\theta$ and $z>1$.

In this work we carry out the corresponding analysis in flat-ring coordinates. ``Flat-rings'' now play the role of
tori.
In Section~\ref{sec2} we consider flat-ring coordinates in the form used by Poole
\cite{Poole2, Poole1}. This form is the most convenient for our purposes.
In Section~\ref{sec3} we solve the Laplace equation by ${\mathcal R}$-separation of variables. The differential equation for $u_1$ and $u_2$ is the Lam\'e equation. In Section~\ref{sec4} we collect properties of simply-periodic Lam\'e functions that we need to introduce
flat-ring harmonics in Section~\ref{sec5}. In Section~\ref{sec5} we prove the main results of this paper: we solve the Dirichlet problem
on flat-rings (Theorem \ref{5:Dirichlet}), we find an integral representation of external flat-ring harmonics in terms of internal
flat-ring harmonics (Theorem \ref{5:intrep}), and we find the series expansion of the fundamental solution in terms of internal and external flat-ring harmonics (Theorem \ref{5:main}).
In Sections~\ref{sec6} and~\ref{sec7} we connect flat-ring with toroidal coordinates and
we compare the corresponding expansions of the fundamental solution.

\section{Flat-ring coordinates}\label{sec2}
\subsection{Flat-ring coordinates in algebraic from}
Flat-ring coordinates $\mu$, $\rho$, $\phi$ form an orthogonal coordinate system in $\R^3$ with rotational symmetry.
We define these coordinates by
\begin{equation}\label{2:coord1}
x=\frac{\cos\phi}{T},\qquad y=\frac{\sin\phi}{T},\qquad z=\frac{1}{T}\bigg(\frac{-\mu\rho}{a}\bigg)^{1/2},
\end{equation}
where
\begin{equation}\label{2:T}
 T=\bigg(\frac{(a-\mu)(a-\rho{)}}{a(a-1)}\bigg)^{1/2} + \bigg(\frac{(1-\mu)(1-\rho)}{a-1}\bigg)^{1/2}.
\end{equation}
The parameter $a>1$ is given. The coordinates $\mu$, $\rho$ vary
according to
\[
-\infty<\mu<0<\rho<1 .
\]
The roots appearing in \eqref{2:coord1}, \eqref{2:T} are understood as positive roots.
Miller \cite[system (15), p.~210]{Miller} considers flat-ring coordinates $\mu'$, $\rho'$, $\phi$ depending on a parameter $a'$ in a slightly different form. The connection is given by
\[
a'=\frac{a}{a-1},\qquad \mu=a-(a-1)\mu',\qquad \rho=a-(a-1)\rho',
\]
where
\[
1<\rho'< \frac{a}{a-1} <\mu'<\infty .
\]

For $\phi=0$ we obtain a planar coordinate system
\begin{equation}\label{2:planar}
 x=\frac{1}{T},\qquad z=\frac{1}{T}\bigg(\frac{-\mu\rho}{a}\bigg)^{1/2}.
\end{equation}
This defines a function $g$ by setting $(x,z)=g(\mu,\rho)$. We introduce the sets
\[
A=(-\infty,0)\times(0,1),\qquad Q=\big\{(x,z)\colon x,z>0,\, x^2+z^2<1\big\}.
\]

\begin{prop}\label{2:prop}
$g$ maps $A$ bijectively onto $Q$. If $(x,z)=g(\mu,\rho)$ then
\begin{equation}\label{2:coordlines}
\frac{\big(x^2+z^2+1\big)^2}{\tau-a}-\frac{\big(x^2+z^2-1\big)^2}{\tau-1}-\frac{4z^2}{\tau}=0\qquad\text{for}\quad \tau=\mu, \rho,
\end{equation}
so the inverse map $(\mu,\rho)=g^{-1}(x,z)$ is given by
\begin{equation}\label{2:inverse}
 \mu=\frac{1}{8x^2}\big({-}\omega-\big(\omega^2+64a x^2z^2\big)^{1/2}\big),\qquad \rho=\frac{1}{8x^2}\big({-}\omega+(\omega^2+64ax^2z^2\big)^{1/2}\big),
 \end{equation}
 where
\[
\omega=4(1+a)z^2+a\big(x^2+z^2-1\big)^2-\big(x^2+z^2+1\big)^2 .
\]
The map $g$ and its inverse map are real-analytic.
\end{prop}
\begin{proof}
Let $(\mu,\rho)\in A$ and $(x,z)=g(\mu,\rho)$. Then
\[
 \frac1T\bigg(1-\frac{\mu\rho}{a}\bigg)= \bigg(\frac{(a-\mu)(a-\rho)}{a(a-1)}\bigg)^{1/2} - \bigg(\frac{(1-\mu)(1-\rho)}{a-1}\bigg)^{1/2},
\]
which leads to
\begin{equation}\label{2:eq1} \frac{x^2+z^2-1}{2x}=\frac12\bigg(\frac1T-\frac1T\frac{\mu\rho}{a}-T\bigg)=
-\bigg(\frac{(1-\mu)(1-\rho)}{a-1}\bigg)^{1/2}.
\end{equation}
By a similar calculation we obtain
\begin{equation}\label{2:eq2}
 \frac{x^2+z^2+1}{2x}=\bigg(\frac{(a-\mu)(a-\rho)}{a(a-1)}\bigg)^{1/2}.
\end{equation}
By \eqref{2:eq1} and~\eqref{2:eq2},
\[
 \frac{\big(x^2+z^2+1\big)^2}{\mu-a}-\frac{\big(x^2+z^2-1\big)^2}{\mu-1}-\frac{4z^2}{\mu} =4x^2\bigg( \frac{\rho-a}{a(a-1)}-\frac{\rho-1}{a-1}+\frac{\rho}{a}\bigg)=0 .
\]
With the same calculation we show that \eqref{2:coordlines} also holds for $\tau=\rho$.

If $(x,z)=g(\mu,\rho)$ with $(\mu,\rho)\in A$ then it is clear that $x,z>0$, and \eqref{2:eq1} shows that $(x,z)\in Q$.
Conversely, suppose that $(x,z)\in Q$. We consider the quadratic polynomial
\begin{equation}\label{2:F}
 F(\tau)=\tau(\tau-1)(\tau-a)\bigg( \frac{\big(x^2+z^2+1\big)^2}{\tau-a}-\frac{\big(x^2+z^2-1\big)^2}{\tau-1}-\frac{4z^2}{\tau}\bigg) .
\end{equation}
The factor of $\tau^2$ is
\[
\big(x^2+z^2+1\big)^2-\big(x^2+z^2-1\big)^2-4z^2= 4x^2>0 .
\]
Moreover,
\begin{gather*}
F(0)= -4a z^2<0,
\\
F(1)= (a-1)\big(x^2+z^2-1\big)^2>0.
\end{gather*}
It follows that there is a unique $(\mu,\rho)\in A$ such that
\begin{equation}\label{2:eq3}
 F(\tau)=4x^2(\tau-\mu)(\tau-\rho)\qquad\text{for}\quad \tau\in\R.
\end{equation}
By setting $\tau=0,1,a$ in \eqref{2:eq3} it follows that
\begin{gather*}
4x^2 \mu\rho= -4a z^2,
\\
4x^2 (1-\mu)(1-\rho)= (a-1)\big(x^2+z^2-1\big)^2,
\\
4x^2 (a-\mu)(a-\rho)= a(a-1)\big(x^2+z^2+1\big)^2.
\end{gather*}
Then we obtain
\begin{gather*}
 T=\frac{x^2+z^2+1}{2x}-\frac{x^2+z^2-1}{2x}=\frac1x\qquad
\text{and}\qquad
z^2=x^2\frac{-\mu\rho}{a}=\frac{1}{T^2}\frac{-\mu\rho}{a}.
\end{gather*}
It follows that $g(\mu,\rho)=(x,z)$ and so we have shown that $g$ is surjective.

Suppose that $g(\mu,\rho)=g(\tilde \mu,\tilde \rho)=(x,z)$. By \eqref{2:coordlines}, $\mu$, $\rho$, $\tilde \mu$, $\tilde \rho$ are all roots of $F(\mu)=0$ with~$F$ defined in \eqref{2:F}.
Since $(\mu,\rho), (\tilde \mu,\tilde \rho)\in A$, this implies $(\mu,\rho)=(\tilde \mu,\tilde \rho)$. Thus $g$ is injective.
It is clear that $g$ and $g^{-1}$ are analytic maps.
\end{proof}

Some coordinate lines are shown in Figure \ref{figure1}. Note that the boundary of the region $Q$ is given by the quarter circle $\rho=1$, the vertical segment $\mu=-\infty$, and two horizontal segments $[0,b]$, $[b,1]$ represented by $\rho=0$ and $\mu=0$ respectively, where
\[
b=\frac{\sqrt{a}-1}{\sqrt{a-1}}=\frac{\sqrt{a-1}}{\sqrt{a}+1}.
\]

\begin{figure}[ht]
\begin{center}
\includegraphics[scale=0.4]{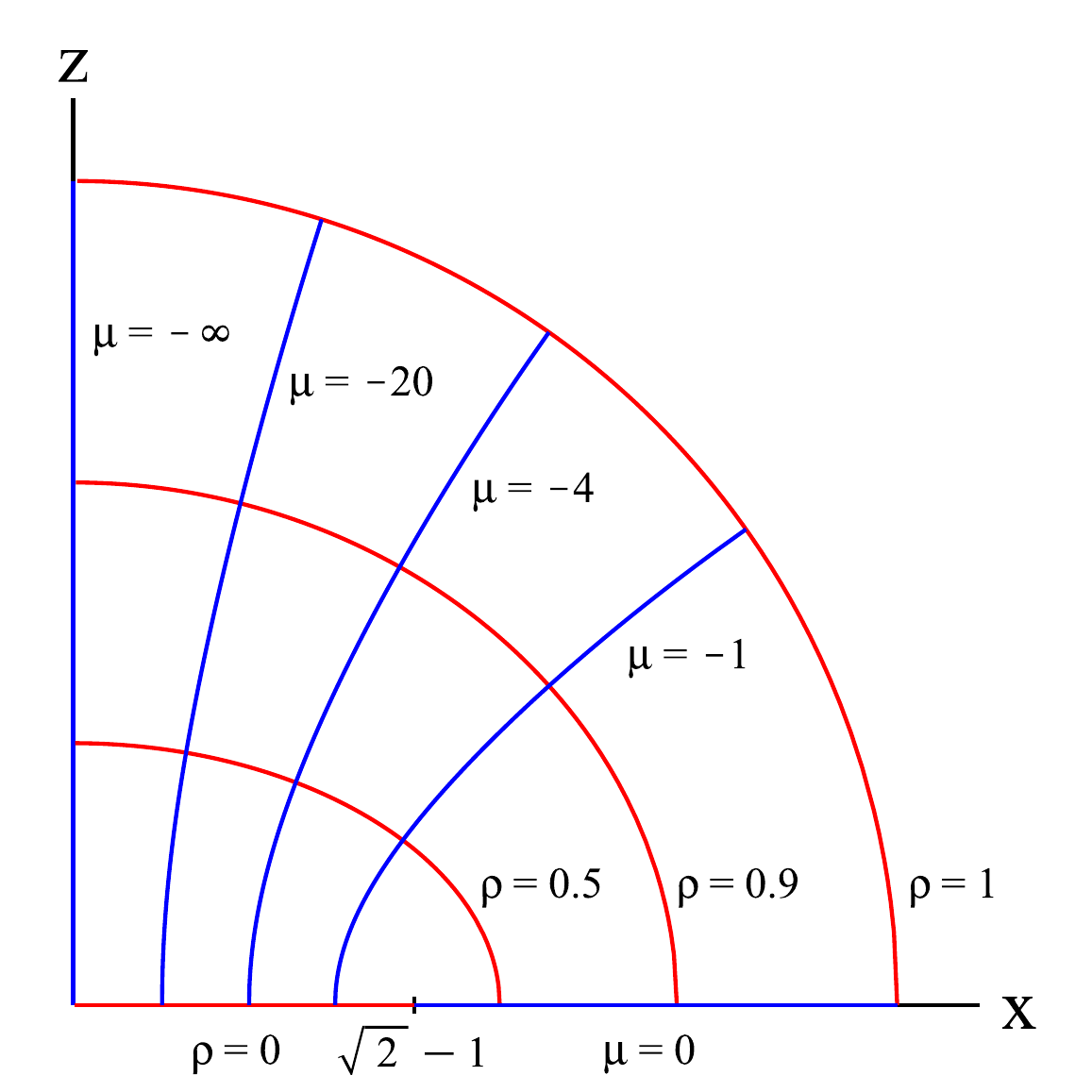}
\caption{Coordinate lines of system \eqref{2:planar} for $a=2$.
\label{figure1}}
\end{center}
\end{figure}

If we change the signs of the square roots in \eqref{2:coord1} and \eqref{2:T} we obtain coordinates for other regions in $\mathbb{R}^2$ according to Table~\ref{t1}.
\begin{table}[ht]
\begin{center}
\caption{Signs of roots in regions of $(x,z)$-plane.}
\label{t1}
\vspace{2mm}
{\small\begin{tabular}{c|c|c|c}
\hline
 &\rule{0pt}{19pt} $\left(\dfrac{(a-\mu)(a-\rho{)}}{a(a-1)}\right)^{1/2}$ &$\left(\dfrac{(1-\mu)(1-\rho)}{a-1}\right)^{1/2} $ &$\left(\dfrac{-\mu\rho}{a}\right)^{1/2}_{\vphantom{q_q}}$
 \\
 \hline
 $x>0,\ z>0, \ x^2+z^2<1$ & $+$& $+$&$+$\\
 $x>0, \ z>0, \ x^2+z^2>1$ & $+$&$-$&$+$\\
 $x>0,\ z<0, \ x^2+z^2>1$ & $+$& $-$& $-$\\
 $x>0, \ z<0, \ x^2+z^2<1$ & $+$& $+$&$-$\\
 \hline
\end{tabular}}
\end{center}
\end{table}

\subsection{Flat-ring coordinates in transcendental form}
We write $a=k^{-2}$ with $k\in(0,1)$, and set
\begin{equation}\label{2:subs}
\rho=\sn^2(s,k),\qquad \mu=\sn^2({\rm i} t,k),
\end{equation}
where $\sn(u,k)$ denotes the {Jacobian} elliptic function of modulus $k$ (see \cite[Chapter~XXII]{WhitWatson}, \cite[Chapter~22]{NIST:DLMF}). We will also use the {Jacobian} elliptic functions $\cn(z,k)$, $\dn(z,k)$,
the complementary modulus $k'=\sqrt{1-k^2}$, and the complete elliptic integrals of the first kind
\[
K=\int_0^{\pi/2} \frac{{\rm d}\theta}{\sqrt{1-k^2\sin^2\theta}},\qquad
K'=\int_0^{\pi/2} \frac{{\rm d}\theta}{\sqrt{1-k'^2\sin^2\theta}}.
\]
By \eqref{2:subs}, $s\in(0,K)$ is mapped bijectively onto $\rho\in(0,1)$, and
$t\in (0,K')$ is mapped bijectively onto $\mu\in(-\infty,0)$.
Substituting \eqref{2:subs} into \eqref{2:coord1} and \eqref{2:T}, we obtain
\begin{equation}\label{2:transc}
x=\frac{\cos\phi}{T} ,\qquad
y=\frac{\sin\phi}{T},\qquad
z=\frac{-{\rm i} k\, \sn(s,k)\sn({\rm i}t,k)}{T},
\end{equation}
where
\[
T=\frac{1}{k'} \dn(s,k)\dn({\rm i}t,k)+\frac{k}{k'} \cn(s,k)\cn({\rm i}t,k) .
\]
By Proposition \ref{2:prop}, the coordinate surface $s=s_0$ is given by
\begin{equation*}
 \frac{k^2\big(x^2+y^2+z^2+1\big)^2}{\dn^2(s_0,k)} - \frac{\big(x^2+y^2+z^2-1\big)^2}{\cn^2(s_0,k)}+ \frac{4z^2}{\sn^2(s_0,k)} =0,
\end{equation*}
and the coordinate surface $t=t_0$ is given by
\begin{equation}\label{2:coordsurface2}
 \frac{k^2\big(x^2+y^2+z^2+1\big)^2}{\dn^2({\rm i} t_0,k)} - \frac{\big(x^2+y^2+z^2-1\big)^2}{\cn^2({\rm i} t_0,k)}+ \frac{4z^2}{\sn^2({\rm i} t_0,k)} =0.
\end{equation}

When $\phi=0$ we have the planar coordinate system
\begin{equation}\label{2:planar2}
x=\frac1T ,\qquad
z=\frac{-{\rm i} k \sn(s,k)\sn({\rm i}t,k)}{T}.
\end{equation}
We note that the coordinates $s$, $t$ can also be written in complex form \cite[equation (8)]{Poole1} as
\begin{gather*}
x+{\rm i} z=f(s+{\rm i}t),\qquad
f(w)=\frac{1}{k'}(\dn(w,k)-k\,\cn(w,k)) .
\end{gather*}
The corresponding metric coefficients
\begin{gather*}
h_s= \bigg(\bigg(\frac{\partial x}{\partial s}\bigg)^2
+\bigg(\frac{\partial z}{\partial s}\bigg)^2\bigg)^{1/2},\qquad
h_t= \bigg(\bigg(\frac{\partial x}{\partial t}\bigg)^2
+\bigg(\frac{\partial z}{\partial t}\bigg)^2\bigg)^{1/2},
\end{gather*}
are \cite[equation (9)]{Poole1}
\begin{equation}\label{2:metric}
 h_s=h_t=\frac{k}{T}\big(\sn^2(s,k)-\sn^2({\rm i}t,k)\big)^{1/2}.
\end{equation}

If we take $s\in(0,2K)$, $t\in(0,K')$ then we obtain a coordinate system for the first quadrant $x,z>0$.
If the point $(x,z)$ is given by coordinates $s$, $t$ then the point{\samepage
\begin{gather*}
(\tilde x,\tilde z)=\bigg(\frac{x}{x^2+z^2},\frac{z}{x^2+z^2}\bigg),
\end{gather*}
obtained from $(x,z)$ by inversion at the unit circle, has coordinates $\tilde s=2K-s$, $\tilde t=t$.}

There are three different ways to extend the coordinates \eqref{2:planar2} to the half-plane $x>0$.

1) We take $s\in (-2K,2K)$ and $t\in(0,K')$. The reflection $(x,z)\mapsto (x,-z)$ is given by $s\mapsto -s$.
Then
\eqref{2:planar2} establishes a bijective map between $(s,t)$ and the region,
\begin{equation}\label{2:Q1}
 Q_1=\{(x,z)\colon x>0\}\setminus \{(x,0)\colon x\ge b \},\qquad b=\frac{1-k}{k'},
\end{equation}
which together with its inverse is real-analytic. $Q_1$ is the right-hand half-plane with a cut from $x=b$ to $x=+\infty$ along the $x$-axis.
In order to express $s$, $t$ in terms of $x$, $z$
form $\mu$, $\rho$ according to~\eqref{2:inverse}
and set
\begin{gather*}
s_1=\arcsn\big(\sqrt\rho,k\big),\qquad t_1=\arcsc\big(\sqrt{-\mu},k'\big),
\end{gather*}
using the inverse Jacobian elliptic functions from \cite[Section~22.15]{NIST:DLMF}.
Then the flat-ring coordinates $s$, $t$ of $x$, $z$ are given by $t=t_1$ and
\begin{gather*}
s=\begin{cases} s_1 & \text{if}\quad R^2+z^2\le 1\quad\text{and}\quad z\ge 0,\\
2K-s_1 & \text{if}\quad R^2+z^2>1\quad\text{and}\quad z\ge 0,\\
-s_1 & \text{if}\quad R^2+z^2\le 1\quad\text{and}\quad z<0,\\
s_1-2K & \text{if}\quad R^2+z^2>1\quad\text{and}\quad z< 0.
\end{cases}
\end{gather*}
Some coordinate lines of this coordinate system are depicted in Figure~\ref{2:fig2}.
\begin{figure}[ht]
\begin{center}
\includegraphics[clip=true,trim={0 0 0 0},width=8cm]{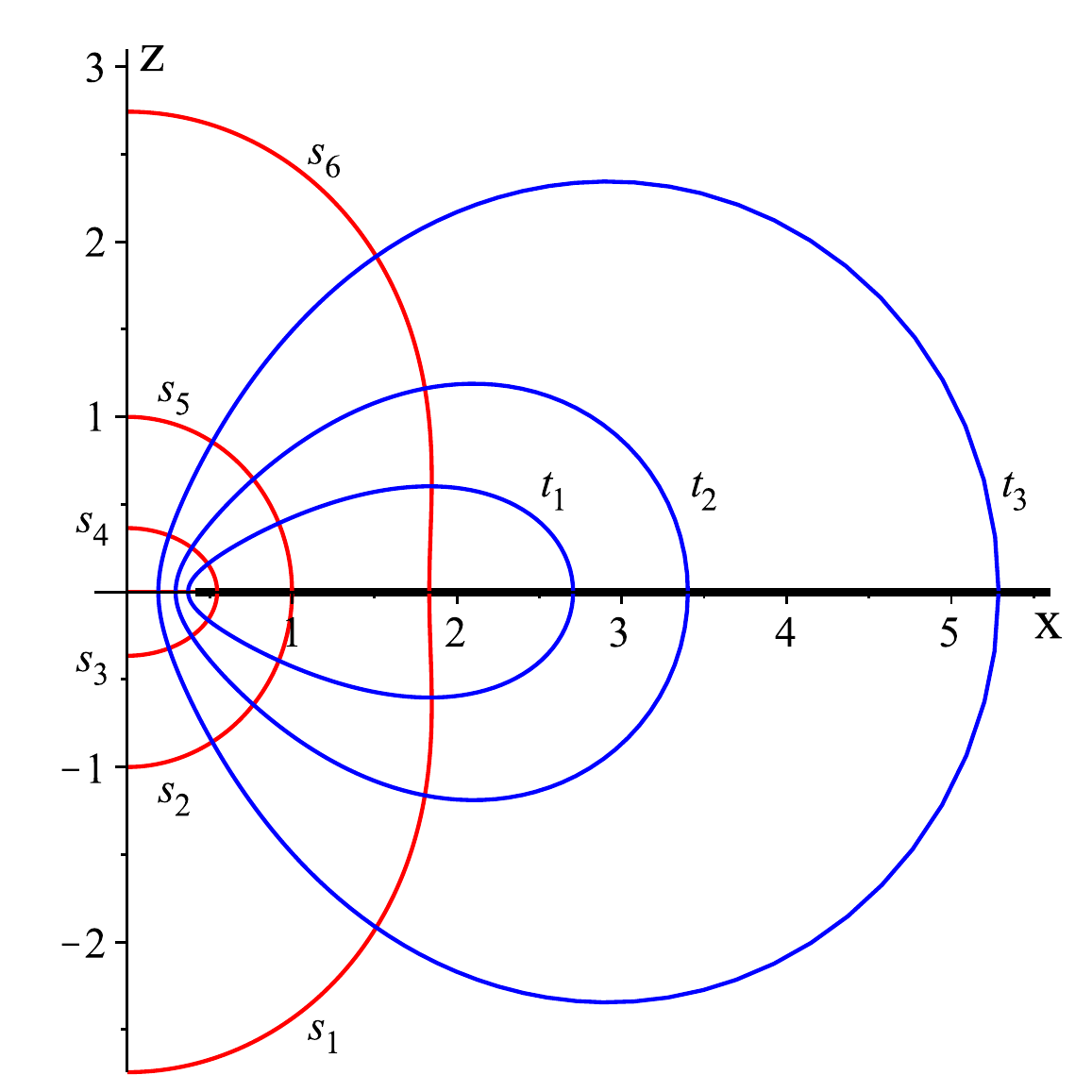}
\caption{Coordinate lines $s=s_i$, $s_1=-1.5K$, $s_2=-K$, $s_3=-0.5K$, $s_4=0.5K$, $s_5=K$, $s_6=1.5K$ and $t=t_i$, $t_1=0.3K'$, $t_2=0.5K'$, $t_3=0.7K'$ of system \eqref{2:planar2} for $a=2$.\label{2:fig2}}
\end{center}
\end{figure}

A second variant is as follows.

2)
We take $s\in(0,2K)$ and $t\in(-K',K')$.
The reflection $(x,z)\mapsto (x,-z)$ is given by $t\mapsto -t$.
Then \eqref{2:planar2} establishes a bijective map between $(s,t)$ and the region,
\begin{equation*}
 Q_2=\{(x,z)\colon x>0\}\setminus \big\{(x,0)\colon 0<x\le b \text{ or } x\ge b^{-1} \big\}.
\end{equation*}
$Q_2$ is the right-hand half-plane with cuts along the $x$-axis from $x=0$ to $x=b$ and from $b^{-1}$ to~$\infty$.

3)
We take $s\in(0,4K)$ and $t\in(0,K')$. The reflection $(x,z)\mapsto (x,-z)$ is given by $s\mapsto 4K-s$.
Then \eqref{2:planar2} establishes a bijective map between $(s,t)$ and the region,
\[
Q_3=\{(x,z)\colon x>0\}\setminus \big\{(x,0)\colon 0<x\le b^{-1} \big\}.
\]

\section[R-separation of the Laplace equation]
{$\boldsymbol{{\mathcal R}}$-separation of the Laplace equation}\label{sec3}

Let $u(x,y,z)$ be a function defined in the region
\begin{equation}\label{3:tildeQ1}
 \tilde{Q}_1=\{(x\cos\phi,x\sin\phi ,z)\colon (x,z)\in Q_1,\, \phi\in(-\pi,\pi) \}
\end{equation}
with $Q_1$ from \eqref{2:Q1}.
 Note that $\tilde{Q}_1$ consists of all of $\R^3$ with the exception of the exterior of a~disk in the $xy$-plane centered at the origin with radius~$b$, and the half-plane $y=0$, $x\leq0$.
In~this region, we have the first variant of transcendental flat-ring coordinates with $s\in(-2K,2K)$, $t\in(0,K')$, $\phi\in(-\pi,\pi)$.

Using cylindrical coordinates
\[
x=r\cos\phi,\qquad y=r\sin\phi,\qquad z=z,
\]
we obtain
\begin{equation*}
 r^{1/2}\Delta u=\bigg[\frac{\partial^2}{\partial r^2}+\frac{\partial^2}{\partial z^2}+\frac1{r^2}\bigg(
 \frac{\partial^2}{\partial\phi^2}+\frac14\bigg)\bigg]\big(r^{1/2}u\big) .
\end{equation*}
Transforming to coordinates $s$, $t$ we find \cite[equation (10)]{Poole1}
\begin{equation}\label{3:pde2}
r^{5/2}M\Delta u=\bigg[\frac{\partial^2}{\partial s^2}+\frac{\partial^2}{\partial t^2}+M\bigg(\frac{\partial^2}{\partial\phi^2}+\frac14\bigg)\bigg]\big(r^{1/2} u\big),
\end{equation}
where
\[
M=k^2\big(\sn^2(s,k)-\sn^2({\rm i}t,k)\big).
\]
Substituting
\[
r^{1/2}u=u_1(s)u_2({\rm i}t)u_3(\phi)
\]
we obtain the following theorem.

\begin{thm}\label{3:t1}
Suppose that $u$ has the form
\begin{equation}\label{3:sepsol}
 u(x,y,z)=\big(x^2+y^2\big)^{-1/4} u_1(s)u_2({\rm i}t)u_3(\phi),
\end{equation}
where $s\in(-2K,2K)$, $t\in(0,K')$, and $\phi\in(-\pi,\pi)$.
If $u_1$ and $u_2$ satisfy the Lam\'e equation
\begin{equation}\label{3:Lame1}
\frac{{\mathrm d}^2w}{{\rm d}\zeta^2}+\bigg(h-\bigg(m^2-\frac14\bigg)k^2\sn^2(\zeta,k)\bigg)w=0,
\end{equation}
and $u_3$ satisfies
\begin{equation} \label{3:fourier}
\frac{{\mathrm d}^2u_3}{{\rm d}\phi^2}+m^2u_3 =0
\end{equation}
for some constants $m$, $h$, then $u(x,y,z)$ satisfies $\Delta u=0$.
\end{thm}

Note that $u_1$ and $u_2$ satisfy the same differential equation but on different intervals $\zeta\in(-2K,2K)$ and $\zeta\in (0,{\rm i}K')$, respectively. Since the function $\sn^2(\zeta,k)$ is analytic on the strip $|\operatorname{Im} \zeta|<K'$, $u_1$, $u_2$ extend analytically to this strip.

\section{Lam\'e functions of the first and second kind}\label{sec4}

In this section we consider particular solutions of the Lam\'e equation that are needed to introduce flat-ring harmonics
in Section \ref{sec5}.
The Lam\'e equation is
\begin{equation}\label{4:Lame}
\frac{{\mathrm d}^2{\sf E}}{{\mathrm d}s^2}+\big(h-\nu(\nu+1)k^2\sn^2(s,k)\big) {\sf E}=0 ,
\end{equation}
where $\nu\ge -\frac12$ and the modulus $k$ of the {Jacobian} elliptic function $\sn$ satisfies $0<k<1$.
The simply-periodic
Lam\'e functions are those functions
which are solutions of the Lam\'{e} equation~\eqref{4:Lame} with
period $4K$. We refer to these as Lam\'e functions
of the first kind.
In the application to flat-ring coordinates $k$ is fixed ($k$ is determined by the coordinate system)
while $\nu=-\tfrac12,\frac12,\frac32,\dots$.
The constant $h$ is the eigenvalue parameter.
The function $\sn^2(s,k)$ is even and has period $2K$ so that \eqref{4:Lame} is a special case of an even Hill equation \cite[Section~28.29({\rm i})]{NIST:DLMF}.
The eigenvalue problem asks for values of~$h$ such that~\eqref{4:Lame} has a nontrivial solution ${\sf E}(s)$ with period~$4K$.
This eigenvalue problem is treated in \cite[equation~(15.5.1)]{ErdelyiHTFIII} and \cite[Theorem~1.2]{MagnusWinkler66}.
The eigenvalue problem splits into four regular Sturm--Liouville problems with separated boundary conditions
\cite[Chapter~VI]{Walter98} according to Table~\ref{4:table2}.
\begin{table}[ht]
\setlength{\arraycolsep}{2mm}
\renewcommand{\arraystretch}{1.2}
\begin{center}
\caption{Notation of simply-periodic Lam\'e functions. \label{4:table2}}
\vspace{2mm}
{\small\begin{tabular}{|l|l|l|l|}
\hline
\multicolumn{1}{|c|}{\rule{0pt}{12pt} boundary conditions$\vphantom{\big)_{q_q}}$} & \multicolumn{1}{c|}{eigenvalues} & \multicolumn{1}{c|}{eigenfunctions} & \multicolumn{1}{c|}{period}
\\
\hline
 ${\sf E}(0)={\sf E}(K)=0$ & $b_\nu^{2n+2}(\kk)$ & $\Es_\nu^{2n+2}(s,\kk)$ & ${\sf E}(s+2K)={\sf E}(s)$ \\
 ${\sf E}'(0)={\sf E}(K)=0$& $b_\nu^{2n+1}(\kk)$ & $\Es_\nu^{2n+1}(s,\kk)$ & ${\sf E}(s+2K)=-{\sf E}(s)$ \\
 ${\sf E}(0)={\sf E}'(K)=0$ & $a_\nu^{2n+1}(\kk)$ & $\Ec_\nu^{2n+1}(s,\kk)$ & ${\sf E}(s+2K)=-{\sf E}(s)$ \\
 ${\sf E}'(0)={\sf E}'(K)=0$& $a_\nu^{2n}(\kk)$ & $\Ec_\nu^{2n}(s,\kk)$ & ${\sf E}(s+2K)={\sf E}(s)$\\
 \hline
\end{tabular}}
\end{center}
\end{table}

\noindent
In this table, in each case $n\in\N_0$ denotes the number of zeros of the eigenfunction in the interval $(0,K)$.
All eigenfunctions are normalized so that
\[
\int_0^K {\sf E}(s)^2\,{\mathrm d}s =1 .
\]

\begin{thm}\label{4:t1}
Each of the systems
\[
\big\{\Es_\nu^{2n+2}\big\}_{n=0}^\infty,\qquad
\big\{\Es_\nu^{2n+1}\big\}_{n=0}^\infty,\qquad
\big\{\Ec_\nu^{2n+1}\big\}_{n=0}^\infty,\qquad
\big\{\Ec_\nu^{2n}\big\}_{n=0}^\infty,
\]
forms an orthonormal basis in the Hilbert space $L^2(0,K)$.
All four systems combined form an orthonormal basis of $L^2(-2K,2K)$ provided each function is multiplied by $\frac12$.
\end{thm}

In our application of simply-periodic Lam\'e functions to flat-ring coordinates we also
need some estimates for these functions that we provide in the following.

\begin{lemma}\label{4:l1}
We have
\begin{gather*}
\frac{\pi^2n^2}{4K^2}\le a_\nu^n(\kk)\le \nu(\nu+1)k^2+\frac{\pi^2n^2}{4K^2}\qquad\text{if}\quad \nu\ge 0,
\\
\nu(\nu+1)k^2+\frac{\pi^2n^2}{4K^2} \le a_\nu^n(\kk)\le \frac{\pi^2n^2}{4K^2}\qquad\text{if}\quad -\frac12\le \nu<0.
\end{gather*}
The same estimates hold with $b_\nu^{n+1}$ in place of $a_\nu^n$.
\end{lemma}
\begin{proof}
The stated inequalities follow from
\begin{gather*}
0\le \nu(\nu+1)k^2\sn^2(s,k)\le \nu(\nu+1)k^2 \qquad \text{if}\quad \nu\ge 0,
\\
\nu(\nu+1)k^2\le \nu(\nu+1)k^2\sn^2(s,k)\le 0 \qquad \text{if}\quad -\frac12\le \nu<0,
\end{gather*}
using a comparison theorem for eigenvalues \cite[Section~27.IX]{Walter98}.
\end{proof}

The next lemma is a refinement of a result in \cite[p.~288]{CourantHilbert1968}.

\begin{lemma}\label{4:l2}
Let $q\colon [a,b]\to\R$ be a continuous function, $\lambda$ a real number, and $y\colon [a,b]\to \R$ a~solution of the
differential equation
\begin{equation}\label{4:SLode}
-y''+q(s)y =\lambda y
\end{equation}
such that $y(a)y'(a)=y(b)y'(b)=0$, and $\int_a^b y(s)^2\,\mathrm{d}s=1$.
Then
\[
y(s)^2\le 2(b-a)^{-1} + 2(R(q))^{1/2}\qquad \text{for}\quad s\in[a,b],
\]
where
\[
R(q):=\max\{q(s)\colon a\le s\le b\}-\min\{q(s)\colon a\le s\le b\}.
\]
\end{lemma}
\begin{proof}
We set
\[
u=\min\{q(s)\colon a\le s\le b\},\qquad v=R(q).
\]
Obviously, the lemma is true when $v=0$ so we assume that $v>0$.
Since we may add the same constant to $q$ and $\lambda$, we also assume that $u>0$.
Multiplying \eqref{4:SLode} by $y$ and integrating from $a$ to $b$ we obtain
\[
\int_a^b y'(s)^2\,{\mathrm d}s +\int_a^b q(s)y(s)^2\,{\mathrm d}s =\lambda .
\]
It follows that
\begin{equation}\label{4:est2}
\int_a^b y'(s)^2\,{\mathrm d}s \le \lambda -u \le \lambda.
\end{equation}
Multiplying \eqref{4:SLode} by $y'$ and integrating from $s$ to $t$, we obtain
\begin{equation}\label{4:est3}
 y'(t)^2+\lambda y(t)^2=2f(s,t)+y'(s)^2 +\lambda y(s)^2 ,
 \end{equation}
where
\[ f(s,t):=\int_s^t q(r)y(r)y'(r)\,{\mathrm d}r. \]
Integrating \eqref{4:est3} from $s=a$ to $s=b$, we find
\begin{equation}\label{4:est4}
 y'(t)^2+\lambda y(t)^2=(b-a)^{-1}\bigg(2\int_a^b f(s,t)\,{\mathrm d}s +\int_a^b y'(s)^2\,{\mathrm d}t+\lambda\bigg) .
 \end{equation}
Using \eqref{4:est2} we have
\[
|f(s,t)|\le(u+v) \bigg(\int_a^b y'(s)^2\,{\mathrm d}t\bigg)^{1/2}\le (u+v)(\lambda-u)^{1/2}.
\]
Therefore, \eqref{4:est2} and~\eqref{4:est4} give
\[
\lambda y(t)^2\le y'(t)^2+\lambda y(t)^2\le 2(u+v)(\lambda-u)^{1/2}+(b-a)^{-1}2\lambda .
\]
By \eqref{4:est2}, $\lambda\ge u>0$, so we can divide by $\lambda$:
\[
y(t)^2\le 2(b-a)^{-1} + 2(u+v)\lambda^{-1}(\lambda-u)^{1/2}.
\]
Now
\[
\lambda^{-1}(\lambda-u)^{1/2}\le \frac12 u^{-1/2}\qquad\text{for}\quad \lambda\ge u,
\]
so we get
\[
y(t)^2 \le 2(b-a)^{-1} + (u+v)u^{-1/2}.
\]
Our $v>0$ is given but we can choose $u>0$ as we like.
Taking $u=v$ we get the desired estimate.
\end{proof}

\begin{lemma}\label{4:l3}
We have
\[
\left\{\Ec_\nu^{n}(s,\kk)\right\}^2\le \frac{2}{K}+2k|\nu|^{1/2}(\nu+1)^{1/2}\qquad\text{for}\quad n\in\N_0,\quad \nu\ge -\frac12, \quad s\in\R.
\]
The inequality also holds with $\Es_\nu^{n+1}$ in place of $\Ec_\nu^n$.
\end{lemma}

\begin{proof}
If $s\in[0,K]$ this follows from Lemma \ref{4:l2} with $q(s)=\nu(\nu+1)k^2\sn^2(s,k)$.
Then $R(q)=|\nu|(\nu+1)k^2$.
Since the simply-periodic Lam\'e functions are even or odd and have period or half-period $2K$, the stated inequality holds
for all $s\in\R$.
\end{proof}

Next, we need an estimate for the functions $\Ec_\nu^n({\rm i}t)$ and $\Es_\nu^{n+1}({\rm i}t)$ for $-K'<t<K'$.
Setting $W(t)={\sf E}({\rm i}t)$ and using \cite[Section~22.6({\rm i}v)]{NIST:DLMF}, \eqref{4:Lame} becomes the modified Lam\'e equation
\begin{equation}\label{4:modlame}
\frac{{\mathrm d}^2W}{{\mathrm d}t^2}-\left(h+\nu(\nu+1)k^2\ssc^2(t,k')\right)W=0 .
\end{equation}
In \eqref{4:modlame} $\ssc(t,k'):=\sn(t,k')/\cn(t,k')$ employs Glaisher's notation \cite[Section~22.2]{NIST:DLMF}.

It should be noted that $W(t)=\Ec_\nu^{n}({\rm i}t,\kk)$ is a solution of \eqref{4:modlame} with $h=a_\nu^n$.
If $n$ is even then $W(t)$ is a real-valued even solution. If $n$ is odd then ${\rm i}W(t)$ is a real-valued odd solution.
A similar statement holds for $\Es_\nu^{n+1}$.

The following lemma is a minor modification of Lemma 5.7 in \cite{CohlVolkcyclide1}.

\begin{lemma}\label{4:l4}
Let $u\colon [0,b]\to \R$ be a solution of the differential equation
\[
u''(t)=q(t)u(t),\qquad t\in[0,b],
\]
determined by the initial conditions $u(0)=1$, $u'(0)=0$ or $u(0)=0$, $u'(0)=1$, where $q$: $[0,b]\to\R$ is a continuous function.
Suppose that $q(t)\ge0$ on $[0,b]$ and $q(t)\ge \lambda^2$ on $[c,b]$ for some $\lambda>0$ and $c\in[0,b)$.
Then $u(t)>0$ and $u'(t)\ge 0$ for $t\in(0,b]$, and
\[
\frac{u(t)}{u(b)}\le 2\expe^{-\lambda(b-c)}\qquad\text{for all}\quad t\in[0,c].
\]
\end{lemma}
\begin{proof}
Since $q(t)\ge 0$ it follows from the initial conditions that $u(t)>0, u'(t)\ge 0$ for $t\in(0,b]$.
The function $z=u'/u$ satisfies the Riccati equation
$z'+z^2=q(t)$, and we know that $z(c)\ge 0$.
Therefore,
\[
z(t)\ge \lambda\tanh(\lambda(t-c))\qquad\text{for}\quad t\in[c,b].
\]
Integrating from $t=c$ to $t=b$ gives
\[
\log \frac{u(b)}{u(c)}\ge \log\cosh(\lambda(b-c))\ge \log\bigg(\frac12\expe^{\lambda(b-c)}\bigg),
\]
which yields the claim since $u$ is nondecreasing.
\end{proof}

\begin{lemma}\label{4:l5}
Let $0<c<b<K'$. Then there is a constant $p\in(0,1)$ independent of $\nu$, $n$ and~$t$ such that
\[
0\le \frac{\Ec_\nu^n({\rm i}t,\kk)}{\Ec_\nu^n({\rm i}b,\kk)}\le 2 p^{n+\nu}\qquad\text{for}\quad \nu>0,\quad n\in\N_0,\quad t\in[0,c].
\]
The same result is true for $\Es_\nu^{n+1}$ in place of $\Ec_\nu^n$.
\end{lemma}
\begin{proof}
Apply Lemma \ref{4:l4} with $q(t)=a_\nu^n(\kk)+\nu(\nu+1)\ssc^2(t,k')$.
Since $\nu>0$, Lemma \ref{4:l1} yields $a_\nu^n(\kk)\ge 0$ and so $q(t)\ge 0$.
Let $\lambda^2=q(c)>0$. Then $q(t)\ge \lambda^2$ on $[c,b]$.
Since
\[
\lambda^2=q(c)\ge \frac{\pi^2n^2}{4K^2}+\nu(\nu+1)k^2\ssc^2(c,k')\ge \gamma^2 (n+\nu)^2,
\]
where
\[
\gamma=2^{-1/2}\min\bigg\{\frac{\pi}{2K},\,k\ssc(c,k')\bigg\}.
\]
This implies the assertion with $p=\expe^{-\gamma (b-c)}$.
\end{proof}

Unfortunately, the case $\nu=-\frac12$ is not covered by Lemma \ref{4:l5}.
It is not even clear that $\Ec_{-1/2}^n({\rm i}t,\kk)\ne 0$ for $0<t<K'$.
To deal with this case we apply the transformation
\[
W(t)=\cd^{1/2}(t,k')\, w(t)
\]
to \eqref{4:modlame}, and obtain
\begin{equation}\label{4:transmodlame}
\frac{{\mathrm d}^2w}{{\mathrm d}t^2}-k^2\frac{\sn(t,k')}{\cn(t,k')\dn(t,k')}\frac{\mathrm{d}w}{{\mathrm d}t}-Q(t)w=0,
\end{equation}
where
\[
Q(t)=h-\frac14k^2+\big(\nu+\tfrac12\big)^2k^2\ssc^2(t,k')+\frac34 k^2\nd^2(t,k') .
\]
Equation \eqref{4:transmodlame} has regular singularities at $t=\pm K'$ with exponents $\pm\big(\nu+\frac12\big)$ at both points.

\begin{lemma}\label{4:l6}
If $\nu>0$, $n\in\N_0$, the function $|\Ec_\nu^n({\rm i}t,\kk)|$ is increasing for $0<t<K'$ and diverges to infinity as $t\to K'$.
If $-\frac12\le \nu\le 0$, $n\in\N_0$, the function $\dc^{1/2}(t,k')|\Ec_\nu^n({\rm i}t,\kk)|$ is increasing for $0<t<K'$ and diverges to infinity as $t\to K'$.
In particular, $\Ec_\nu^n({\rm i}t,\kk)\ne 0$ for $0<t<K'$.
The same statements hold with $\Es_\nu^{n+1}$ in place of $\Ec_\nu^n$.
\end{lemma}
\begin{proof}
If $\nu>0$ we consider differential equation \eqref{4:modlame} with $h=a_\nu^n(\kk)$.
By Lemma \ref{4:l1}, $a_\nu^n(\kk)+\nu(\nu+1)k^2\ssc^2(t,k')>0$ for $0<t<K'$. This implies that $|\Ec_\nu^n({\rm i}t,\kk)|$ is increasing. It diverges to infinity
because $t=K'$ is a regular singularity of this equation with exponents $-\nu$ and $\nu+1$.
If $-\frac12\le\nu<0$ we consider differential equation
\eqref{4:transmodlame} with $h=a_\nu^n(\kk)$. By Lemma \ref{4:l1},
\[
Q(t)\ge a_\nu^n(\kk)-\frac14k^2+\frac34k^2\ge \nu(\nu+1)k^2+\frac12k^2>0.
\]
It follows that the solutions of \eqref{4:transmodlame} determined by initial conditions
$w(0)=1$, $w'(0)=0$ or~$w(0)=0$, $w'(0)=1$ are positive and increasing on $(0,K')$. If $-\frac12<\nu<0$ we see as before
that $w(t)$ diverges to infinity as $t\to K'$. This is also true for $\nu=-\frac12$ by noting that the unique solution $w$
of \eqref{4:transmodlame} with $\nu=-\frac12$ which is analytic at $t=K'$ and satisfies $w(K')=1$ has $w'(K')=0$ and
$w''(K')=\frac12 a_{-1/2}^n(\kk)-\frac18k^2+\frac{3}{8}>0$.
 \end{proof}

\begin{lemma}\label{4:l7}
Let $0<b<K'$. Then there are constants $C>0$ and $p\in(0,1)$ independent of $n$ and $t$ such that
\[
0\le \frac{\Ec_{-\frac12}^n({\rm i}t,\kk)}{\Ec_{-\frac12}^n({\rm i}b,\kk)}\le C p^n\qquad\text{for}\quad n\in\N_0, \quad t\in[0,b].
\]
The same result is true for $\Es_\nu^{n+1}$ in place of $\Ec_\nu^n$.
\end{lemma}
\begin{proof}
We apply Lemma \ref{4:l4} with $q(t)=a_\nu^n(\kk)-\frac14k^2\ssc^2(t,k')$.
By Lemma \ref{4:l1},
\[
q(t)\ge -\frac14 k^2+\frac{\pi^2n^2}{4K^2}-\frac14 k^2\ssc^2(b,k')\qquad \text{for}\quad 0\le t\le b.
\]
If $n$ is sufficiently large then
\[
q(t)\ge \bigg(\frac{\pi}{2K} n- B\bigg)^2\qquad \text{for}\quad 0\le t\le b,
\]
where $B$ is a positive constant independent of $n$ and $t$.
Then Lemma \ref{4:l4} implies the desired estimate for large $n$.
Using Lemma \ref{4:l6}, we see that by choosing $C$ large enough the estimate is true for all $n\in\N_0$.
\end{proof}

The Lam\'e equation \eqref{4:Lame} has regular singularities at $s=\pm {\rm i}K'$ with exponents $\nu+1$, $-\nu$.
The Lam\'e function of the second kind $\Fc_\nu^n(s,\kk)$, $|\operatorname{Im} s|<K'$, is defined as a solution
of \eqref{4:Lame} with $h=a_\nu^n(\kk)$ belonging to the exponent $\nu+1$, that is,
\begin{equation}\label{4:F}
 \Fc_\nu^n(s,\kk)= \sum_{p=0}^\infty d_p (s-{\rm i}K')^{2p+\nu+1},\qquad d_0\ne 0,
\end{equation}
for $s$ close to ${\rm i}K'$.
Similarly, $\Fs_\nu^{n+1}$ is the
solution of \eqref{4:Lame} with $h=b_\nu^{n+1}$
which belongs to the exponent $\nu+1$ at $s={\rm i}K'$.
It follows from Lemma \ref{4:l6} that $\Ec_\nu^n$, $\Fc_\nu^n$ and $\Es_\nu^n$, $\Fs_\nu^n$ are linearly independent.
Therefore, we can normalize the Lam\'e functions of the second kind such that
\begin{equation}\label{4:wronskian}
{\sf F}({\rm i}t)\frac{{\mathrm d}}{{\mathrm d}t} {\sf E}({\rm i}t)-{\sf E}({\rm i}t)\frac{{\mathrm d}}{{\mathrm d}t} {\sf F}({\rm i}t)=1,
\end{equation}
where ${\sf E}=\Ec$, ${\sf F}=\Fc$ or ${\sf E}=\Es$, ${\sf F}=\Fs$.

\section{Flat-ring harmonics}\label{sec5}
\subsection{Internal flat-ring harmonics}\label{sec5.1}

We use the first variant of flat-ring coordinates $s$, $t$, $\phi$
with $s\in(-2K,2K)$, $t\in(0,K')$
and $\phi\in(-\pi,\pi)$. If $t_0\in(0, K')$ is a fixed value, then the coordinate surface \eqref{2:coordsurface2}
describes a closed surface, a flat-ring as depicted in Figure \ref{flatring3d}.
Also, in Figure \ref{flatring3d2} we present
a three-dimensional visualization of the flat-rings and their orthogonal surfaces
which we refer to as peanuts.

\begin{figure}[ht]
\centering
 \includegraphics[scale=.9]{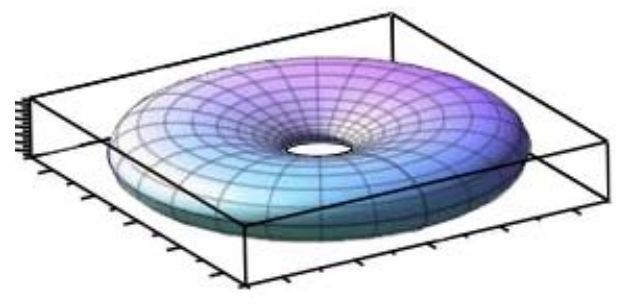}
 \put(-24,29){\makebox(0,0)[lb]{$-2$}}
 \put(-54,22){\makebox(0,0)[lb]{$-1$}}
 \put(-80,15){\makebox(0,0)[lb]{$0$}}
 \put(-110,8){\makebox(0,0)[lb]{$1$}}
 \put(-142,1){\makebox(0,0)[lb]{$2$}}
 \put(-190,2){\makebox(0,0)[lb]{$-2$}}
 \put(-208,13){\makebox(0,0)[lb]{$-1$}}
 \put(-220,26){\makebox(0,0)[lb]{$0$}}
 \put(-240,39){\makebox(0,0)[lb]{$1$}}
 \put(-258,50){\makebox(0,0)[lb]{$1$}}
 \put(-286,63){\makebox(0,0)[lb]{$-0.3$}}
 \put(-277,85){\makebox(0,0)[lb]{$0.3$}}
 \caption{Flat-ring, {$k=0.7$, $t=0.2 K'$}.\label{flatring3d}}
\end{figure}

\begin{figure}[hb]
\centering
 \includegraphics[clip=true,trim={1.5cm 0.5cm 1.5cm 0.5cm},width=12cm]{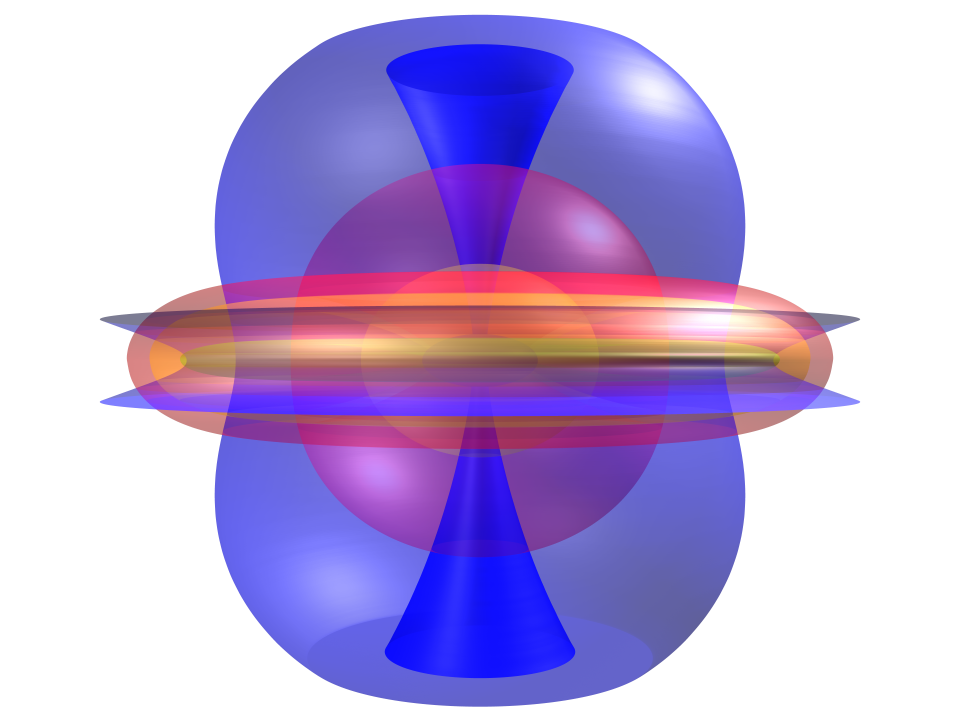}
 \caption{{For $k=0.7$ this figure depicts a three-dimensional visualization of rotationally-invariant flat-ring cyclides for $t\in\{\pm 0.10K',\pm 0.28K',\pm 0.35K',\pm 0.90K'\}$ and orthogonal peanut shaped cyclides $s\in\{0.369K,0.923K,1.26K,1.48K,1.75K\}$}. {Note that the ring cyclides are only partially shown for $t=\pm 0.1K'$ and the peanut cyclide is only partially shown for $s=1.75K$ (since these would extend beyond the figure).} \label{flatring3d2}}
\end{figure}

The interior of this flat-ring is given by
\begin{equation}\label{5:flatring}
 D_1=\bigg\{ \r\in\R^3\colon \frac{k^2(\|\r\|^2+1)^2}{\dn^2({\rm i}t_0,k)} - \frac{(\|\r\|^2-1)^2}{\cn^2({\rm i}t_0,k)}+ \frac{4z^2}{\sn^2({\rm i}t_0,k)} > 0\bigg\},
\end{equation}
where $\r=(x,y,z)$, $\|\r\|=\big(x^2+y^2+z^2\big)^{1/2}$.
Internal flat-ring harmonics are harmonic functions of the separated form~\eqref{3:sepsol} which are harmonic in the region $D_1$ for every $t_0\in(0,K')$. Therefore, internal flat-ring harmonics have to be harmonic
on all of $\R^3$ except the $z$-axis.

Suppose that $u_1(\zeta)$ and $u_2(\zeta)$ are solutions of the Lam\'e equation \eqref{3:Lame1} for $|\operatorname{Im}\zeta|<K'$,
and $u_3$ is a solution of \eqref{3:fourier}.
By Theorem \ref{3:t1}, $u(x, y, z)$ defined by \eqref{3:sepsol} is a harmonic function in $\tilde Q_1$ as defined in \eqref{3:tildeQ1}. We want this function to be harmonic on $\R^3$ except the $z$-axis. Clearly, we need $m\in\Z$, and we choose $u_3(\phi)=\expe^{{\rm i}m\phi}$ (alternatively, we could use $\cos(m\phi)$, $m=0,1,2,\dots$ and $\sin(m\phi)$, $m=1,2,3,\dots$). Then we have to require that the function $u_1(s)u_2({\rm i}t)$ is analytic in the right-hand half plane $x>0$, $z\in \R$. We know that $u_1(s)u_2({\rm i}t)$ is always analytic in the quadrant $x>0$, $z>0$. When we analytically extend this function
to the quadrant $x>0, z<0$ across the segment $(x,0)$, $0<x<b$, by using the first variant of flat-ring coordinates
then this extension has the value $u_1(-s)u_2({\rm i}t)$ at the point $(x,-z)$ when $(x,z)$ has coordinates
$s$, $t$. When we extend across the segment $(x,0)$, $b^{-1}<x$, using the third variant we obtain $u_1(4K-s)u_2({\rm i}t)$.
We want these extensions to be the same, so we need $u_1$ to be
periodic with period $4K$ (provided $u_2$ is not identically zero).
Therefore, we take $u_1=\Ec_{|m|-\frac12}^n$ or $u_1=\Es_{|m|-\frac12}^{n+1}$. When we analytically extend across the segment
$(x,0)$, $b<x<b^{-1}$, using the second variant of flat-ring coordinates we obtain $u_1(s)u_2(-{\rm i}t)$.
This extension should also be the same as the previous two, so we have to require that $u_1$ and $u_2$ have the same parity.
Therefore, $u_1$ has to be a constant multiple of~$u_2$. We can take $u_1=u_2$.
Thus we are led to define internal flat-ring harmonics by
\begin{gather}
\Gc^n_m(x,y,z)=\big(x^2+y^2\big)^{-1/4} \Ec^n_{|m|-\frac12}(s,\kk)\Ec^n_{|m|-\frac12}({\rm i}t,\kk)\expe^{{\rm i}m\phi}, \label{5:int1}
\\
\Gs^{n+1}_m(x,y,z)=\big(x^2+y^2\big)^{-1/4} \Es^{n+1}_{|m|-\frac12}(s,\kk)\Es^{\,n+1}_{|m|-\frac12}({\rm i}t,\kk)\expe^{{\rm i}m\phi}, \label{5:int2}
\end{gather}
where $m\in\Z$, $n\in\N_0$.

We collect some properties of internal harmonics in the following theorem.
We will use the Kelvin transformation \cite[Chapter~IX.2]{Kellogg}. If $u(\r)$ is a harmonic function then
its Kelvin transformation
\[
\tilde u(\r)=\|\r\|^{-1}u(\sigma(\r))
\]
is also harmonic, where
$\sigma$ denotes the inversion at the unit sphere
\[
\sigma(\r)=\|\r\|^{-2}\r .
\]

\begin{thm}\label{5:int}
The internal flat-ring harmonics $\Gc_m^n$ and $\Gs_m^{n+1}$ are harmonic functions defined in all of $\R^3$ except for the $z$-axis.
They have the following symmetry properties
\begin{gather*}
\Gc^n_m(\sigma(\r))=\|\r\|\Gc^n_m(\r),
\\
\Gs^{n+1}_m(\sigma(\r))=-\|\r\|\Gs^{n+1}_m(\r),
\\
\Gc^n_m(x,y,-z)=(-1)^n\Gc^n_m(x,y,z),
\\
\Gs^{n+1}_m(x,y,-z)=(-1)^n \Gs^{n+1}_m(x,y,z).
\end{gather*}
\end{thm}
\begin{proof}
From our discussion at the beginning of this section we know that $u=\Gc$ and $u=\Gs$ are harmonic functions on $\R^3$ except for the $z$-axis and the two circles centered at the origin with radii $b^{-1}$ and $b$ in the $xy$-plane, respectively. The function $u$ is bounded in a neighborhood of the two circles. Therefore, the two circles are removable singularities of $u$ (see \cite[Theo\-rem~XIII, p.~271]{Kellogg}). 
Hence, $u$ is harmonic on $\R^3$ except for along the $z$-axis.

The inversion $\sigma$ in the upper half-space $z>0$ is expressed in flat-ring coordinates by the map $s\mapsto 2K-s$.
The Lam\'e functions $\Ec$ remain unchanged under this reflection while the Lam\'e functions $\Es$ change sign.
Similarly, the reflection $z\mapsto -z$ is expressed by the map $s\mapsto-s$.
The Lam\'e functions $\Ec^n_\nu$, $\Es^{n+1}_\nu$ remain unchanged under this reflection if $n$ is even but change sign when $n$ is odd.
\end{proof}

\subsection{The Dirichlet problem}
Theorem \ref{4:t1} implies the following theorem.
\begin{thm}\label{5:basis}
The system of functions
\[
(8\pi)^{-1/2} \Ec^{\,n}_{|m|-\frac12}(s,\kk)\expe^{{\rm i}m\phi},\qquad
(8\pi)^{-1/2} \Es^{\,n+1}_{|m|-\frac12}(s,\kk)\expe^{{\rm i}m\phi},\qquad m\in\Z,\quad n\in\N_0,
\]
is an orthonormal basis in the Hilbert space
\[
H_1=L^2((-2K,2K)\times (-\pi,\pi)).
\]
\end{thm}

We use internal flat-ring harmonics to solve the following Dirichlet problem.
For a fixed $t_0\in(0,K')$ consider the region $D_1$ defined in \eqref{5:flatring}.
For a given function $f$ defined on the boundary $\partial D_1$ of $D_1$ we find the function $u$
which is harmonic on $D_1$ and attains the boundary values $f$ on $\partial D_1$ in the weak sense.
The latter means that $\big(x^2+y^2\big)^{1/4}u$ (expressed in terms of flat-ring coordinates $s$, $t$, $\phi$) evaluated at $t_1\in(0,t_0)$ converges to $\big(x^2+y^2\big)^{1/4}f$ in the Hilbert space $H_1$ as $t_1\to t_0$.
Since $D_1$ has a positive distance to the $z$-axis the factor $\big(x^2+y^2\big)^{1/4}$ can be omitted in this definition.

We note that the solution of this Dirichlet problem is unique ({\rm i}f it exists).
To see this, assume that $u$ is a harmonic function in $D_1$ that attains the boundary values $0$ in the weak sense.
Let $g(s,t,\phi)$ be the function obtained from $\big(x^2+y^2\big)^{1/4}u(x,y,z)$ by expressing $x$, $y$, $z$
in flat-ring coordinates~\eqref{2:transc}.
This function is analytic for $s,\phi\in\R$ and $-t_0<t<t_0$.
For $t\in(0,t_0)$, $m\in\Z$, $n\in\N_0$ define
\[ v({\rm i}t)=\int_{-\pi}^\pi\int_{-2K}^{2K} g(s,t,\phi)\Ec_{|m|-\frac12}^n(s,\kk) \expe^{-{\rm i}m\phi}\,{\mathrm d}s\,{\mathrm d}\phi.\]
By differentiation under the integral sign, integration by parts and using \eqref{3:pde2} we find that $v$ satisfies
Lam\'e's equation \eqref{4:Lame} with $h=a^n_{|m|-\frac12}(\kk)$.
Since $g(s,t,\phi)=g(-s,-t,\phi)$, $v$ is even if $E$ is even and odd if $E$ is odd.
Therefore, $v$ is a constant multiple of~$E$.
By assumption, $v({\rm i}t)\to 0$ as $t\to t_0$
but this is impossible by Lemma~\ref{4:l6} unless $v({\rm i}t)$ is identically zero.
Using the same argument with $\Es$ in place of $\Ec$ we conclude that $u$ is identically zero by Theorem~\ref{5:basis}.

\begin{thm}\label{5:Dirichlet}
Let $f$ be a function defined on the boundary $\partial D_1$ of the region $D_1$ for some $t_0\in(0,K')$.
Suppose that $f$ is represented in the first variant of flat-ring coordinates as
\[
\big(x^2+y^2\big)^{1/4}f(x,y,z)=g(s,\phi) ,\qquad s\in(-2K,2K),\quad \phi\in(-\pi,\pi),
\]
such that $g\in H_1$.
For all $m\in\Z$ and $n\in\N_0$, define
\begin{gather*}
 c_m^n:=\frac{1}{8\pi\Ec^n_{|m|-\frac12}({\rm i}t_0,\kk)}\int_{-\pi}^\pi \int_{-2K}^{2K} g(s,\phi)\Ec^n_{|m|-\frac12}(s,\kk)\expe^{-{\rm i}m\phi}\,{\mathrm d}s\,{\mathrm d}\phi
\\ \hphantom{ c_m^n\,}
{} = \frac{1}{8\pi\bigl\{\Ec^n_{|m|-\frac12}({\rm i}t_0,\kk)\bigr\}^2}\int_{\partial D_1} \frac{1}{h_s(\r)} f(\r) \Gc_{-m}^n(\r)\,{\mathrm d}{S}(\r),
\end{gather*}
and
\begin{gather*}
 d_m^{n+1}:=\frac{1}{8\pi\Es^{\,n+1}_{|m|-\frac12}({\rm i}t_0,\kk)} \int_{-\pi}^\pi \int_{-2K}^{2K} g(s,\phi)\Es^{\,n+1}_{|m|-\frac12}(s,\kk)\expe^{-{\rm i}m\phi}{\mathrm d}s\,{\mathrm d}\phi
 \\ \hphantom{d_m^{n+1}\,}
{} = \frac{1}{8\pi\bigl\{\Ec^{\,n+1}_{|m|-\frac12}({\rm i}t_0,\kk)\bigr\}^2}\int_{\partial D_1} \frac{1}{h_s(\r)} f(\r) \Gs_{-m}^{n+1}(\r){\mathrm d}S(\r),
\end{gather*}
where $h_s$ is the metric coefficient \eqref{2:metric}. Then the function
\begin{equation}\label{5:sol}
 u(\r)=\sum_{m\in\Z}\sum_{n=0}^\infty \left(c_m^n\Gc_m^n(\r)+d_m^{n+1}\Gs_m^{n+1}(\r)\right)
\end{equation}
is harmonic in $D_1$ and it attains the boundary values $f$ on $\partial D_1$ in the weak sense.
The infinite series in \eqref{5:sol} converges absolutely and uniformly in compact subsets of $D_1$.
\end{thm}
\begin{proof}
Using flat-ring coordinates, we can write surface integrals over $\partial{D_1}$ as double integrals:
\[
 \int_{\partial D_1} f(\r)\,{\mathrm d}{S}(\r)= \int_{-\pi}^\pi\int_{-2K}^{2K} h_s h_\phi f\,{\mathrm d}s\,{\mathrm d}\phi
\]
with the metric coefficient {$h_s$ from \eqref{2:metric} and} $h_\phi=\big(x^2+y^2\big)^{1/2}$.
When we replace $f$ by $f \Gc/h_s$ or $f\Gs/h_s$ this shows that the two formulas given for $c_m^n$ and $d_m^{n+1}$ agree.

Let $s\in(-2K,2K)$, $\phi\in(-\pi,\pi)$ and $0<t\le t_1<t_0$. Using Lemmas~\ref{4:l3}, \ref{4:l5} and~\ref{4:l7}
we estimate
\[
\left|\frac{\Ec_{|m|-\frac12}^n({\rm i}t,\kk)}{\Ec_{|m|-\frac12}^n({\rm i}t_0,\kk)}
\Ec_{|m|-\frac12}^n(s,\kk)\expe^{{\rm i}m\phi}\right|\le
C p^{|m|+n} (1+|m|)^{1/2},
\]
where the constants $C$ and $p\in(0,1)$ are independent
of $m$, $n$, $s$, $t$, $\phi$.
We have a similar estimate with $\Es$ in place of $\Ec$. Since $\{c_m^n\}$ and $\{d_m^n\}$ are bounded double sequences, this proves that the
series in \eqref{5:sol} is absolutely and uniformly convergent
on compact subsets of $D_1$. Consequently, by Theorem \ref{5:int}, $u$ defined by \eqref{5:sol} is harmonic
in $D_1$.

Let $\tilde u(t)$ be the function $(x^2+y^2)^{1/4}u$ for a given $t\in(0,t_0)$ considered as a function of $(s,\phi)$.
Computing the norm $\|\tilde u-g\|$ in the Hilbert space $H_1$ by the Parseval
equality, we obtain
\begin{gather*}
 \|\tilde u-g\|^2\le 8\pi\sum_{m\in\Z}\sum_{n=0}^\infty
\Bigg(|c_m^n|^2
\Bigg(1-\frac{\Ec_{|m|-\frac12}^n({\rm i}t,\kk)}{\Ec_{|m|-\frac12}^n({\rm i}t_0,\kk)}\Bigg)^2
+|d_m^{{n+1}}|^2
\Bigg(1-\frac{\Es_{|m|-\frac12}^{n+1}({\rm i}t,\kk)}{\Es_{|m|-\frac12}^{n+1}({\rm i}t_0,\kk)}\Bigg)^2
\Bigg).
\end{gather*}
It is easy to see that the right-hand side converges to $0$ as $t\to t_0$.
It follows that $u$ attains the boundary values $f$ in the weak sense.
\end{proof}

It is well-known \cite[Chapter~XI]{Kellogg} that the classical Dirichlet problem on the flat-ring domain~$D_1$ has a unique solution, that is, for a given
continuous function $f$ on $\partial D_1$ there is a~unique continuous function
$u$ defined on $\overline{D}_1$ which is harmonic in $D_1$ and agrees with $f$ on~$\partial D_1$.
Obviously, for continuous $f$, the classical solution of the Dirichlet problem agrees with the solution furnished by Theorem \ref{5:Dirichlet}.

\subsection{External flat-ring harmonics}
External flat-ring harmonics are harmonic functions $u$ of the form \eqref{3:sepsol} which are harmonic outside of all
flat-ring regions \eqref{5:flatring}. Therefore, they are harmonic on $\R^3$
except for the annulus $b^2\leq x^2+y^2 \leq b^{-2}$ in the $xy$-plane. It is clear that $m$ must be an integer and arguing as at the beginning of Section~\ref{sec5.1}, we see that $u_1$ must have period $4K$.
We note that, for fixed $s\in(0,2K)$,
\[
\lim_{t\to K'} (K'-t)T\sim \frac1{k'}(\dn(s,k)+\cn(s,k)).
\]
Since the function $u$ has to be harmonic along the $z$-axis, we require that $u_2$ is a solution of the Lam\'e equation~\eqref{3:Lame1} such that $\big(x^2+y^2\big)^{-1/4}u_2({\rm i}t)=T^{1/2} u_2({\rm i}t)$
stays bounded when~$t$ approaches~$K'$. Therefore, $u_2$ has to be one of the Lam\'e functions $\Fc$ and $\Fs$ defined in Section~\ref{sec4}. Thus we define external flat-ring harmonics by
\begin{gather}
\Hc^n_m(x,y,z)=\big(x^2+y^2\big)^{-1/4} \Ec^n_{|m|-\frac12}(s,\kk)\Fc^n_{|m|-\frac12}({\rm i}t,\kk)\expe^{{\rm i}m\phi}, \label{5:ext1}
\\
\Hs^{n+1}_m(x,y,z)=\big(x^2+y^2\big)^{-1/4} \Es^{\,n+1}_{|m|-\frac12}(s,\kk)\Fs^{n+1}_{|m|-\frac12}({\rm i}t,\kk)\expe^{{\rm i}m\phi}, \label{5:ext2}
\end{gather}
where $m\in\Z$, $n\in\N_0$. We recall that we use the first variant of the transcendental flat-ring coordinates.

\begin{thm}
The external flat-ring harmonics $\Hc_m^n$, $\Hs_m^{n+1}$ are harmonic functions defined in all of $\R^3$ except for the closed annulus in the $xy$-plane centered at the origin with inner radius $b=\frac{k'}{1+k}$ and outer radius $b^{-1}$.
Moreover,
\begin{gather}
\Hc^n_m(\sigma(\r))=\|\r\|\Hc^n_m(\r),\label{5:ext3}
\\
\Hs^{n+1}_m(\sigma(\r))=-\|\r\|\Hs^{n+1}_m(\r),\label{5:ext4}
\\
\Hc^n_m(x,y,-z)=(-1)^n\Hc^n_m(x,y,z),\label{5:ext5}
\\
\Hs^{n+1}_m(x,y,-z)=(-1)^n \Hs^{n+1}_m(x,y,z) \label{5:ext6},
\end{gather}
and
\begin{equation}\label{5:ext7}
 \lim_{||\r\|\to\infty} \Hc^n_m(\r)=\lim_{\|\r\|\to\infty} \Hs_m^n(\r)=0 .
 \end{equation}
\end{thm}

\begin{proof}
From our preceding discussion we know that the function $u=\Hc_m^n$ (or $u=\Hs_m^{n+1}$) is harmonic on $\R^3$ except for the $z$-axis and the annulus $b^2\leq x^2+y^2 \leq b^{-2}$ in the $xy$-plane.
It~follows from \eqref{4:F} that $u$ stays bounded when we approach the $z$-axis.
This shows that $u$ is harmonic
on all of $\R^3$ except for the annulus $b^2\leq x^2+y^2 \leq b^{-2}$ in the $xy$-plane.
The symmetry properties are shown as in the proof of Theorem \ref{5:int}. The limits \eqref{5:ext7} follow from
\eqref{5:ext3} and~\eqref{5:ext4} and the fact that external flat-ring harmonics are harmonic at the origin.
\end{proof}

We now show that external harmonics admit an integral representation in terms of internal harmonics.

\begin{thm}\label{5:intrep}
Let $t_0\in(0,K')$, $m\in\Z$, $n\in\N_0$, and let $\r^\ast$ be a point outside $\overline{D}_1$, where $D_1$ is given by \eqref{5:flatring}. Then
\begin{equation}\label{5:intrep1}
 \Hc_m^n(\r^\ast)=\frac{1}{4\pi\big\{\Ec_{|m|-\frac12}^n({\rm i}t_0,\kk)\big\}^2}\int_{\partial D_1} \frac{\Gc_m^n(\r)}{h_s(\r)\|\r-\r^\ast\|}\,\,{\mathrm d}S(\r) ,
\end{equation}
and
\begin{equation}\label{5:intrep2}
 \Hs_m^{n+1}(\r^\ast)=\frac{1}{4\pi\big\{\Es_{|m|-\frac12}^{n+1}({\rm i}t_0,\kk)\big\}^2}\int_{\partial D_1} \frac{\Gs_m^{n+1}(\r)}{h_s(\r)\|\r-\r^\ast\|}\,\,{\mathrm d}S(\r) .
\end{equation}
\end{thm}
\begin{proof}
Let $D$ be an open bounded subset of $\R^3$ with smooth boundary.
For $u,v\in C^2(\overline D)$, Green's formula states that
\begin{equation}\label{5:Green}
\int_D(u\Delta v-v\Delta u)\,{\rm d}\r = \int_{\partial D} \bigg(u\frac{\partial v}{\partial\nu}-v\frac{\partial u}{\partial\nu}\bigg)\,{\mathrm d}S,
\end{equation}
where $\frac{\partial u}{\partial\nu}$ is the outward normal derivative of $u$ on the boundary $\partial D$ of~$D$.
We apply \eqref{5:Green} to $D=D_1$, $u=G=\Gc_m^n$, and
\begin{equation}\label{5:v}
v(\r)=\frac{1}{4\pi\|\r^\ast-\r\|}.
\end{equation}
Since $\Delta u=\Delta v=0$ on $D_1$ we obtain
\begin{equation}\label{5:eq2}
0=\int_{\partial D_1} \left(G \frac{\partial v}{\partial\nu}-v \frac{\partial G}{\partial\nu}\right) \,{\mathrm d}S .
\end{equation}
Let $B_r(\q)$ denote the open ball centered at $\q\in\R^3$ with radius $r>0$.
We apply \eqref{5:Green} a second time with $D=B_R(\0) -\overline D_1-B_\epsilon(\r^\ast)$ with large $R$ and small $\epsilon>0$.
Choose $u=H=\Hc_m^n$ and~$v$ as in \eqref{5:v}. Note that $\Delta u=\Delta v=0$ on $D$.
By a standard argument \cite[Chapter~VII, Proof of Theorem 5.1]{ZachmanoglouThoe76}, taking the limit $\epsilon\to 0$, we obtain
\begin{equation}\label{5:eq3}
H(\r^\ast)=-\int_{\partial B_R(\0)} \bigg(H\frac{\partial v}{\partial\nu}-v \frac{\partial H}{\partial\nu}\bigg) {\mathrm d}S+
\int_{\partial D_1} \bigg(H\frac{\partial v}{\partial\nu}-v \frac{\partial H}{\partial\nu}\bigg){\mathrm d}S,
\end{equation}
where, in the second integral, $\frac{\partial}{\partial\nu}$ denotes again the derivative in the direction of the outward normal as in \eqref{5:eq2}.
As $r=\|\r||\to\infty$, we have
\[
v(\r)=O\big(r^{-1}\big),\qquad \frac{\partial v}{\partial \nu}=O\big(r^{-2}\big)
\]
and, from \eqref{5:ext3},
\[
H(\r)=O\big(r^{-1}\big),\qquad \frac{\partial H}{\partial \nu}=O\big(r^{-2}\big).
\]
Therefore, the first integral on the right-hand side of \eqref{5:eq3} tends to $0$ as $R\to\infty$, so we obtain
\begin{equation}\label{5:eq4}
H(\r^\ast)=\int_{\partial D_1}
\bigg(H\frac{\partial v}{\partial\nu}-v \frac{\partial H}{\partial\nu}\bigg) {\mathrm d}S.
\end{equation}
Set ${\sf E}(\zeta)=\Ec_{|m|-\frac12}^n(\zeta)$ and ${\sf F}(\zeta)=\Fc_{|m|-\frac12}^n(\zeta)$.
We multiply~\eqref{5:eq2} by $-{\sf F}({\rm i}t_0)$, then multiply~\eqref{5:eq4} by ${\sf E}({\rm i}t_0)$ and add these equations. Using the definitions \eqref{5:int1} and \eqref{5:ext1} of internal and external flat-ring harmonics and cancelling terms, we find
\begin{equation}\label{5:eq5}
{\sf E}({\rm i}t_0)H(\r^\ast)=\int_{\partial D_1} v\bigg({\sf F}({\rm i}t_0)\frac{\partial G}{\partial\nu}-{\sf E}({\rm i}t_0)\frac{\partial H}{\partial \nu}\bigg) {\mathrm d}S .
\end{equation}
Since flat-ring coordinates are orthogonal, the normal derivative and the partial derivative with respect to $t$ are related by
\[
\frac{\partial}{\partial \nu}= \frac{1}{h_t} \frac{\partial}{\partial t},
\]
where $h_t$ is given in \eqref{2:metric}.
Let $\r=(x,y,z)\in \partial D_1$ with flat-ring coordinates $s$, $t$, $\phi$.
Then
{\samepage\begin{gather*}
{\sf F}({\rm i}t_0)\frac{\partial G}{\partial\nu}(\r)-{\sf E}({\rm i}t_0)\frac{\partial H}{\partial \nu}(\r)
 \\ \qquad
{} = {\sf F}({\rm i}t_0)\frac{\partial \big(\big(x^2+y^2\big)^{-1/4}\big)}{\partial\nu}{\sf E}({\rm i}t_0){\sf E}(s)\expe^{{\rm i}m\phi}
+{\sf F}({\rm i}t_0)\big(x^2+y^2\big)^{-1/4} h_t^{-1}i{\sf E}'({\rm i}t_0){\sf E}(s)\expe^{{\rm i}m\phi}
\\ \qquad\hphantom{=}
 {}-{\sf E}({\rm i}t_0)\frac{\partial \big(\big(x^2+y^2\big)^{-1/4}\big)}{\partial\nu}{\sf F}({\rm i}t_0){\sf E}(s)\expe^{{\rm i}m\phi}
-{\sf E}({\rm i}t_0)\big(x^2+y^2\big)^{-1/4}h_t^{-1} i{\sf F}'({\rm i}t_0){\sf E}(s)\expe^{{\rm i}m\phi}
\\ \qquad
{}= h_t^{-1} \big(x^2+y^2\big)^{-1/4}i\{{\sf F}({\rm i}t_0){\sf E}'({\rm i}t_0)
-{\sf F}'({\rm i}t_0){\sf E}({\rm i}t_0)\} {\sf E}(s)\expe^{{\rm i}m\phi} .
\end{gather*}}

\noindent
We now use the Wronskian \eqref{4:wronskian} and obtain
\begin{equation}\label{5:eq6}
{\sf F}({\rm i}t_0)\frac{\partial G}{\partial\nu}(\r)-{\sf E}({\rm i}t_0)\frac{\partial H}{\partial \nu}(\r)
= \frac{1}{h_t(\r){\sf E}({\rm i}t_0)}G(\r).
\end{equation}
When we substitute \eqref{5:eq6} in \eqref{5:eq5} we obtain \eqref{5:intrep1}. The proof of \eqref{5:intrep2} is similar.
\end{proof}

\subsection{Expansion of the fundamental solution}
We obtain the expansion of \eqref{5:v} in internal and external flat-ring harmonics by combining Theorems \ref{5:Dirichlet} and \ref{5:intrep}.

\begin{thm}\label{5:main}
Let $\r,\r^\ast\in \R^3$ with flat-ring coordinates $t,t^\ast\in(0,K')$, respectively.
If $t< t^\ast$ then
\begin{equation}\label{5:expansion}
\frac{1}{\|\r-\r^\ast\|}=\frac12\sum_{m\in\Z}\sum_{n=0}^\infty \left(\Gc_m^n(\r)\Hc_{-m}^n(\r^\ast)+\Gs_m^{n+1}(\r)\Hs_{-m}^{n+1}(\r^\ast)\right).
\end{equation}
\end{thm}
\begin{proof}
We pick $t_0$ such that $t<t_0<t^\ast$, and consider the domain $D_1$ defined in \eqref{5:flatring}.
Therefore, by applying Theorem \ref{5:Dirichlet} to the harmonic function $f(\q):=\frac{1}{\|\q-\r^\ast\|}$, we find
\[
 f(\r)=\sum_{m\in\Z}\sum_{n=0}^\infty \big(c_m^n\Gc_m^n(\r)+d_m^{n+1}\Gs_m^{n+1}(\r)\big),
\]
where $c_m^n$ and $d_m^{n+1}$ can be evaluated by Theorem \ref{5:intrep}:
\[
c_n^m=\frac12 \Hc_{-m}^n(\r^\ast),\qquad d_m^{n+1}=\frac12 \Hs_{-m}^{n+1}(\r^\ast).
\]
We use $h_s=h_t$. Thus we obtain \eqref{5:expansion}.
\end{proof}

\subsection{Addition theorem for the azimuthal Fourier coefficients}

Then, the azimuthal Fourier expansion of
$\|{\bf r}-{\bf r}^\ast\|^{-1}$ is well-known
\cite[equation~(15)]{CT}
\begin{equation}
\frac{1}{\|\r-\r^\ast\|}
=\frac{1}{\pi\sqrt{RR^\ast}}
\sum_{m\in\Z} \expe^{{\rm i}m(\phi-\phi^\ast)}
Q_{m-\frac12}(\chi),
\label{oneoraz}
\end{equation}
where
\begin{equation}
\label{chidef}
\chi=\frac{R^2+{R^\ast}^2+(z-z^\ast)^2}{2RR^\ast},
\end{equation}
and $R$, $R^\ast$, $z$, $z^\ast$ are the
radial and vertical rotationally-invariant
cylindrical coordinates of $\r$ and $\r^\ast$
respectively.
It is understood that
\[
\frac1R=\frac1{k'}\dn(s,k)\dn({\rm i}t,k)+\frac{k}{k'}\cn(s,k)\sn({\rm i}t,k), \qquad
z=-{\rm i}kR\sn(s,k)\sn({\rm i}t,k),
\]
and, similarly, $R^\ast$, $z^\ast$ are expressed in terms of $s^\ast$, $t^\ast$.
One can express $\chi$ in a different way as follows.

\begin{lem}\label{l1}
Let $k\in(0,1)$,
$s,s^\ast\in\R$, $t,t^\ast\in(-K',K')$ such that {$0<t<t^\ast<K'$}. Then $\chi$
from \eqref{chidef} expressed in terms
of flat-ring cyclide coordinates
\eqref{2:coord1}
is given by
\begin{gather*}
\chi=k^2\sn(s,k)\sn({\rm i}t,k)\sn(s^\ast,k)\sn({\rm i}t^\ast,k)
 -\frac{k^2}{k'^2}\cn(s,k)\cn({\rm i}t,k)\cn(s^\ast,k)\cn({\rm i}t^\ast,k)
 \\ \hphantom{\chi=}
{}+\frac{1}{k'^2}\dn(s,k)\dn({\rm i}t,k)\dn(s^\ast,k)\dn({\rm i}t^\ast,k).
 \end{gather*}
\end{lem}

\begin{proof}
Setting
\[
\frac{1}{\tilde R}=\frac1{k'}\dn(s,k)\dn({\rm i}t,k)-\frac{k}{k'}\cn(s,k)\sn({\rm i}t,k),
\]
then a short calculation shows that
$R=\tilde R(R^2+z^2).$
Now
\begin{gather*}
\chi=\frac{R^2+{R^\ast}^2+(z-z^\ast)^2}{2RR^\ast}=\frac{\big(R^2+z^2\big)+\big({R^\ast}^2+{z^\ast}^2\big)-2zz^\ast}{2RR^\ast}
= \frac{1}{2\tilde R R^\ast}+\frac{1}{2R\tilde R^\ast}-\frac{zz^\ast}{RR^\ast}.
\end{gather*}
If we express the last expression in terms of $s$, $t$, $s^\ast$, $t^\ast$, the
desired formula is obtained.
\end{proof}

Note that the Legendre functions $Q_\nu:=Q_\nu^0$, $P_\nu{:=}P_\nu^0$
are the associated Legendre func\-ti\-ons~$Q_\nu^\mu$, $P_\nu^\mu$ with $\mu=0$.
These Legendre functions $Q_{m-\frac12}$, for
$m\in\Z$, appear in the separation
of variables for the three-variable Laplace
equation in toroidal coordinates (see Section \ref{sec7} below.
By starting with Theorem \ref{5:main} and substituting
\eqref{5:int1},
\eqref{5:int2},
\eqref{5:ext1} and
\eqref{5:ext2}, we obtain
the following double summation
expression for
$\|{\bf r}-{\bf r}^\ast\|^{-1}$
in flat-ring cyclide coordinates
expressed
in terms of internal and external flat-ring harmonics
\begin{gather}\label{5:expansionb}
\frac{1}{\|\r-\r^\ast\|}=\frac{1}{2\sqrt{RR^\ast}}
\sum_{m\in\Z}\expe^{{\rm i}m(\phi-\phi^\ast)}\nonumber
\\ \hphantom{\frac{1}{\|\r-\r^\ast\|}=}
{}\times\sum_{n=0}^\infty \biggl(\Ec^n_{|m|-\frac12}(s,\kk)
\Ec^n_{|m|-\frac12}(s^\ast,\kk)
\Ec^n_{|m|-\frac12}({\rm i}t,\kk)
\Fc^n_{|m|-\frac12}({\rm i}t^\ast,\kk)\nonumber
\\ \hphantom{\frac{1}{\|\r-\r^\ast\|}=\times\sum_{n=0}^\infty \biggl(}
{}+\Es^{n+1}_{|m|-\frac12}(s,\kk)
\Es^{\,n+1}_{|m|-\frac12}(s^\ast,\kk)
\Es^{\,n+1}_{|m|-\frac12}({\rm i}t,\kk)
\Fs^{n+1}_{|m|-\frac12}({\rm i}t^\ast,\kk)
{\biggr)}.
\label{flatdoub}
\end{gather}
By comparing the azimuthal Fourier coefficients
of \eqref{flatdoub} and \eqref{oneoraz}, one
can obtain the following addition theorem
for the odd-half-integer Legendre function
of the second kind $Q_{m-\frac12}$, expressed
in terms of Lam\'{e} functions of the
first (simply periodic) and second kind (non-periodic)
in flat-ring cyclide coordinates.

\begin{thm}
\label{addnthm}Let $m\in{\N_0}$, $k\in(0,1)$,
$s,s^\ast\in\R$, $t,t^\ast\in(-K',K')$ such that {$0<t<t^\ast<K'$}. Then
\begin{gather*}
 Q_{m-\frac12}(\chi)= {\frac{\pi}{2}}
\sum_{n=0}^\infty \biggl(\Ec^n_{{m}-\frac12}(s,\kk)
\Ec^n_{{m}-\frac12}(s^\ast,\kk)\Ec^n_{{m}-\frac12}({\rm i}t,\kk)
\Fc^n_{{m}-\frac12}({\rm i}t^\ast,\kk) \nonumber
\\ \hphantom{Q_{m-\frac12}(\chi)= {\frac{\pi}{2}}
\sum_{n=0}^\infty \biggl(}
 {}+\Es^{n+1}_{{m}-\frac12}(s,\kk)
\Es^{\,n+1}_{{m}-\frac12}(s^\ast,\kk)
\Es^{\,n+1}_{{m}-\frac12}({\rm i}t,\kk)
\Fs^{n+1}_{{m}-\frac12}({\rm i}t^\ast,\kk)\biggr).
\end{gather*}
\end{thm}

\begin{proof}
Start with
\begin{gather*}
R=\frac{\sqrt{1-k^2}}{k \cn(s,k)\cn({\rm i}t,k)+\dn(s,k)\dn({\rm i}t,k)},\qquad
z=\frac{-{\rm i} k\,\sqrt{1-k^2} \sn(s,k)\sn({\rm i}t,k)}{k\cn(s,k)\cn({\rm i}t,k)+\dn(s,k)\dn({\rm i}t,k)},
\end{gather*}
and similarly for $R^\ast$, $z^\ast$.
Then comparing
\eqref{oneoraz}
to \eqref{flatdoub} identifies the azimuthal Fourier coefficients as a Legendre function of odd-half-integer degree given in terms of two infinite series of products of Lam\'e functions of the first and second kind. This completes the proof.
\end{proof}

\begin{rem}
The above addition Theorem $\ref{addnthm}$ is
a rare instance of an infinite series representation over Lam\'e functions
which reduces to an analytic $($Gauss hypergeometric$)$ function.
\end{rem}

The discovery of the new addition Theorem \ref{addnthm} leads to an integral relation for products of Lam\'e functions of the first and second kind.

\begin{thm}\label{5:intrel}
Let $m,n\in\N_0$, $s^\ast\in\R$, $0<t<t^\ast<K'$.
Then
\[
\int_{-2K}^{2K} Q_{m-\frac12}(\chi) \Ec_{m-\frac12}^n(s)\,{\mathrm{d}}s= 2\pi \Ec^n_{m-\frac12}(s^\ast,k)
\Ec^n_{m-\frac12}({\rm i}t,k)
\Fc^n_{m-\frac12}({\rm i}t^\ast,k)
\]
and
\[
\int_{-2K}^{2K} Q_{m-\frac12}(\chi) \Es_{m-\frac12}^{n+1}(s)\,{\mathrm{d}}s= 2\pi \Es^{n+1}_{m-\frac12}(s^\ast,k)
\Es^{n+1}_{m-\frac12}({\rm i}t,k)\Fs^{n+1}_{m-\frac12}({\rm i}t^\ast,k).
\]
\end{thm}

\begin{proof}
Recall that we proved uniform convergence of the series in Theorem~\ref{addnthm} with respect to $s$ for fixed $s^\ast$, $t$, $t^\ast$.
Therefore, the integral relations follows from the orthogonality of the function system
\[
\Ec_{m-\frac12}^n(s,k),\qquad \Es_{m-\frac12}^{n+1}(s,k),\qquad n\in\N_0,
\]
over the interval $-2K\le s\le 2K$ and the normalization
\[
\int_0^K \big\{\Ec_{m-\frac12}^n(s,k)\big\}^2\,{\mathrm{d}}s=\int_0^K \big\{\Es_{m-\frac12}^{n+1}(s,k)\big\}^2\,{\mathrm{d}}s=1
\]
which leads to
\[
\int_{-2K}^{2K} \big\{\Ec_{m-\frac12}^n(s,k)\big\}^2\,{\mathrm{d}}s=\int_{-2K}^{2K} \big\{\Es_{m-\frac12}^{n+1}(s,k)\big\}^2\,{\mathrm{d}}s=4,
\]
which completes the proof.
\end{proof}

Theorem \ref{5:intrel} is a special case of \cite[Corollary 2.8]{Volkmer84}.
This reference shows that Theorem~\ref{5:intrel} remains true if $m$ is any nonnegative real number
in place of a nonnegative integer.
The statement of Theorem~\ref{5:intrel} suggests an alternative proof of the expansion of a fundamental solution of
the three-variable Laplace equation
in terms of internal and external
flat-ring cyclidic harmo\-nics~\eqref{5:expansion}.
One can then start with the known Theorem~\ref{5:intrel} and
then derive Theorem~\ref{addnthm} from~it. One can then use Theorem \ref{addnthm}
in order to obtain the expansion of a fundamental solution using~\eqref{oneoraz}.

\section{Lemmas on second-order linear differential equations}
\label{sec6}

As we show in Section \ref{sec7}, flat-ring coordinates approach toroidal coordinates in the limit $k\to 0$. Therefore, we expect that special functions found by separation of the Laplace equation in flat-ring coordinates will approach
corresponding special functions found by separation of the Laplace equation in toroidal coordinates. In order to
make this statement precise we introduce the following three lemmas.

\begin{lemma}\label{6:l1}\quad
\begin{enumerate}\itemsep=0pt
\item[$(a)$]
Let $a\in\R$, and let $\{b_n\}$ be a sequence of real numbers such that $a<b_n\to \infty$ as $n\to\infty$.
\item[$(b)$]
For $n\in\N$, let $p_n,q_n\colon [a,b_n]\to\R$ be continuous functions such that $q_n(x)<0$ for all $x\in[a,b_n]$.
For every $n\in\N$, let $y_n\colon [a,b_n]\to\R$ be a nontrivial solution of the differential equation
\begin{equation*}
y_n''+p_n(x)y_n'+q_n(x)y_n =0
\end{equation*}
such that $y_n(b_n)=0$.
\item[$(c)$]
Let $p_\infty, q_\infty\colon [a,\infty)\to\R$ be continuous functions such that
$p_n(x)\to p_\infty(x)$ and $q_n(x)\to q_\infty(x)$ as $n\to\infty$ uniformly on each compact interval $[a,b]$. Suppose that the differential equation
\begin{gather}\label{limitode}
 y_\infty''+p_\infty(x)y'_\infty+q_\infty(x)y_\infty=0
\end{gather}
admits a bounded nontrivial solution $y_\infty\colon [a,\infty)\to\R$, and that every solution of \eqref{limitode} which is linearly independent of $y_\infty$
is unbounded as $x\to\infty$.
\end{enumerate}

Under assumptions $(a)$, $(b)$, $(c)$, we have
\begin{equation}\label{conv1}
 \frac{y_n(x)}{y_n(a)}\to \frac{y_\infty(x)}{y_\infty(a)}\qquad \text{and} \qquad \frac{y_n'(x)}{y_n(a)}\to \frac{y'_\infty(x)}{y_\infty(a)}
\end{equation}
as $n\to\infty$ uniformly on every compact interval $[a,b]$.
The same result is true if the condition $y_n(b_n)=0$ is replaced by $y_n'(b_n)=0$.
\end{lemma}

\begin{proof}
Without loss of generality we assume that $y_n'(b_n)<0$.
Consider $y_n$ for a fixed $n$. There is $\epsilon>0$ such that $y_n(x)>0$, $y_n'(x)<0$ for $x\in(b_n-\epsilon,b_n)$.
Suppose there is $x_0\in[a,b_n)$ such that $y_n'(x_0)=0$ and $y_n'(b_n)<0$ for $x\in(x_0,b_n]$.
This implies the contradiction $0\ge y_n''(x_0)=-q_n(x_0)y_n(x_0)>0$.
Therefore,
\begin{equation}\label{yn}
y_n(x)>0\qquad \text{and}\qquad y_n'(x)<0\qquad \text{for all}\quad x\in[a,b_n).
\end{equation}
Suppose that the sequence $\frac{y_n'(a)}{y_n(a)}$ converges to $-\infty$. Without loss of generality, we assume that $y_n'(a)=-1$ for all $n\in\N$.
Then $y_n(a)\to 0$ as $n\to\infty$.
By assumption $(c)$, it follows that $y_n(x)\to y(x)$ uniformly on $[a,b]$, where $y(x)$ is the solution of \eqref{limitode} with initial conditions
$y(a)=0$, $y'(a)=-1$. This contradicts \eqref{yn}.

Now suppose that the sequence $\frac{y_n'(a)}{y_n(a)}$ converges to $c\in(-\infty,0]$.
Without loss of generality, we assume that $y_n(a)=1$ for all $n\in\N$.
Then $y_n'(a)\to c$ as $n\to\infty$.
It follows that $y_n(x)\to y(x)$ and $y_n'(x)\to y'(x)$ as $n\to\infty$ uniformly on compact intervals $[a,b]$, where $y(x)$ is the solution of~\eqref{limitode} with initial conditions $y(a)=1$, $y'(a)=c$.
It follows from \eqref{yn} that $0\le y(x)\le 1$ for all $x\ge a$. Therefore, $y$ must be a constant multiple of $y_\infty$.
This shows that $c=\frac{y'_\infty(a)}{y_\infty(a)}$ is uniquely determined.
By taking subsequences if necessary, this proves that $\frac{y_n'(a)}{y_n(a)}\to c$
and so~\eqref{conv1} is established.
The proof in the case $y_n'(b_n)=0$ is very similar.
\end{proof}

The following result is well-known.

\begin{lemma}\label{6:l2}
Let $D$ be a simply-connected domain in $\C$, $a\in D$. For $n\in\N$, let $p_n,q_n,p_\infty,q_\infty${\rm :} $D\to\C$ be analytic functions such that
$p_n(z)\to p_\infty(z)$ and $q_n(z)\to q_\infty(z)$ locally uniformly for~$z\in D$. For each $n\in\N$ let $y_n\colon D\to\C$ be a solution
of the differential equation
\[
y_n''+p_n(z)y_n'+q_n(z)y_n=0 ,
\]
and let $y_\infty\colon D\to\C$ be a solution of
\[
y_\infty''+p_\infty(z)y_\infty'+q_\infty(z)y_\infty=0.
\]
If
\[
y_n(a)\to y_\infty(a),\qquad y_n'(a)\to y_\infty'(a)\qquad \text{as}\quad n\to\infty,
\]
then
\[
y_n(z)\to y_\infty(z) \qquad\text{as}\quad n\to\infty\quad \text{locally uniformly for}\quad z\in D.
\]
\end{lemma}

The following lemma can be considered as known although {we} do not have a precise reference. We prove it following \cite[Section~4.5]{Coddington}.

\begin{lemma}\label{6:l3}
Let $D$ be a simply-connected domain in the complex plane $\C$ containing $0$.
Let $p_n\colon D\to\C$ be analytic functions for each $n\in\N\cup\{\infty\}$ such that $p_n(z)\to p_\infty(z)$ locally uniformly
on $D$, and $p_n(0)=\nu(\nu+1)$ for all $n\in\N\cup\{\infty\}$, where $\nu\ge -\frac12$. For each $n\in\N\cup\{\infty\}$, let $u_n\colon D\to\C$ be the unique analytic function such that
$u_n(0)=1$ and $y_n(z):=z^{\nu+1}u_n(z)$ solves
\begin{equation}\label{oden}
 z^2y_n''=p_n(z) y_n.
 \end{equation}
Then $u_n(z)\to u_\infty(z)$ locally uniformly on $D$.
\end{lemma}
\begin{proof}
Without loss of generality we take $D=\{z\in\C\colon |z|<R\}$. The point $z=0$ is a regular singularity of differential equation \eqref{oden}. The exponents are $-\nu$ and $\nu+1$.
Since $\nu\ge -\frac12$, the solutions $y_n$ as stated in the lemma exist.
We write
\begin{gather*}
 p_n(z)=\sum_{k=0}^\infty a_{n,k}z^k,\qquad a_{n,0}=\nu(\nu+1),
 \\
 u_n(z)=\sum_{k=0}^\infty b_{n,k}z^k,\qquad b_{n,0}=1.
\end{gather*}
By substituting $y_n(z)=z^{\nu+1}u_n(z)$ in \eqref{oden}, we obtain
\begin{equation*}
 k(k+2\nu+1)b_{n,k} =\sum_{j=0}^{k-1} a_{n,k-j}b_{n,j} \qquad \text{for}\quad k\in\N.
\end{equation*}
Since $a_{n,k}\to a_{\infty,k}$ as $n\to\infty$ for all $k\in\N_0$, we obtain by induction that
\begin{equation}\label{conv}
b_{n,k}\to b_{\infty,k}\qquad\text{as}\quad n\to\infty\qquad \text{for}\quad k\in\N_0.
\end{equation}

Let $0<r_1<r_2<R$. By assumption, there is a constant $M$ ({\rm i}ndependent of $n$, $z$) such that $|p_n(z)|\le M$ for all $n\in\N$ and $|z|\le r_2$.
By Cauchy's estimate, $|a_{n,k}|\le M r_2^{-k}$ for all $n$, $k$.
Set $\gamma_0=1$ and define $\gamma_k$ for $k\in\N$ recursively by
\begin{equation}\label{recursion2}
 k^2\gamma_k=M\sum_{j=0}^{k-1} r_2^{j-k}\gamma_j.
\end{equation}
By induction, we see that $|b_{n,k}|\le \gamma_k$ for all $n$, $k$.
Replacing $k$ by $k+1$ in \eqref{recursion2}, we find
\[ r_2(k+1)^2 \gamma_{k+1}=\big(M+k^2\big)\gamma_k.\]
The ratio test shows that
\begin{equation}\label{conv2}
\sum_{k=0}^\infty \gamma_k r_1^k<\infty .
\end{equation}
Since $|b_{n,k}z^k|\le \gamma_kr_1^k$ for $|z|\le r_1$ and all $n$, $k$, we obtain from \eqref{conv} and \eqref{conv2} that $u_n(z)\to u_\infty(z)$ as $n\to\infty$ uniformly for $|z|\le r_1$. This completes the proof.
\end{proof}

\section[The limit k to 0 to toroidal coordinates]
{The limit $\boldsymbol{k\to 0}$ to toroidal coordinates}\label{sec7}

\subsection[Flat-ring coordinates in the limit k to 0]
{Flat-ring coordinates in the limit $\boldsymbol {k\to0}$}
\label{toroidal}

In the following we will frequently use that
\begin{equation*}
K(k)\to\frac12\pi,\qquad K'(k)\to \infty\qquad\text{as}\quad k\to 0.
\end{equation*}
Moreover,
\begin{equation}\label{k0}
\sn(\zeta,k)\to\sin \zeta,\qquad \cn(\zeta,k)\to \cos \zeta,\qquad \dn(\zeta,k)\to 1\qquad
\text{as} \quad k\to 0
\end{equation}
locally uniformly on $\C$ \cite[Table~22.5.3]{NIST:DLMF} and
\begin{equation}\label{k1}
 \sn(\zeta,k)\to\tanh \zeta,\qquad \cn(\zeta,k)\to \sech \zeta,\qquad \dn(\zeta,k)\to \sech \zeta\qquad\text{as}\quad k\to 1
\end{equation}
locally uniformly at all points $\zeta$ such that $\cosh \zeta\ne 0$ \cite[Table~22.5.4]{NIST:DLMF}.

We use the first variant of flat-ring coordinates $s\in(-2K,2K)$, $t\in(0,K')$, and set $\tau=K'-t$, $\psi=2K-s$.
Then
\begin{equation*}
 T= \frac{\dn(\psi,k)-\cn(\psi,k)\dn(\tau,k')}{k'\sn(\tau,k')}.
\end{equation*}
Therefore,
\[
\lim_{k\to 0} T=\frac{\cosh\tau-\cos \psi}{\sinh \tau}.
\]
It follows that
\[
\lim_{k\to0} x=\frac{\sinh\tau\cos\phi}{\cosh \tau-\cos \psi},\qquad
\lim_{k\to 0} y=\frac{\sinh\tau\sin\phi}{\cosh \tau-\cos\psi}.
\]
Moreover,
\[
z=\frac{k'\sn(\psi,k)\cn(\tau,k')}{\dn(\psi,k)-\cn(\psi,k)\dn(\tau,k')}
\]
gives
\[
\lim_{k\to 0} z=\frac{\sin\psi}{\cosh\tau-\cos\psi} .
\]
These limits can be taken for $\tau>0$ and $\psi\in(0,2\pi)$.
Therefore, flat-ring coordinates become toroidal coordinates \eqref{1:toroidal} as $k\to0$.

\subsection[Simply-periodic Lam\'e functions in the limit k to 0]
{Simply-periodic Lam\'e functions in the limit $\boldsymbol{k\to 0}$}

\begin{thm}\label{Elimit}
For $\nu\ge -\frac12$, we have
\begin{gather*}
 a_\nu^n(\kk)\to n^2,\qquad n\in\N_0,
 \\
 b_\nu^n(\kk)\to n^2,\qquad n\in\N,
 \\
 \Ec_\nu^0(s,\kk)\to \bigg(\frac{2}{\pi}\bigg)^{1/2},
 \\
 \Ec_\nu^n(s,\kk)\to \bigg(\frac{4}{\pi}\bigg)^{1/2}\cos\bigg(n\bigg(\frac\pi2-s\bigg)\bigg),\qquad n\in\N,
 \\
 \Es_\nu^n(s,\kk)\to \bigg(\frac{4}{\pi}\bigg)^{1/2}\sin\bigg(n\bigg(\frac\pi2-s\bigg)\bigg),\qquad n\in\N
\end{gather*}
as $k\to 0$ locally uniformly for $s\in\C$.
\end{thm}
\begin{proof}
The limits of the eigenvalues follow from
Lemma \ref{4:l1}. The limits of the eigenfunctions follow from Lemma \ref{6:l2}
and \eqref{k0}.
\end{proof}

In the following theorem, $Q$ denotes the Legendre function of the second kind \cite[Chapter~14]{NIST:DLMF}.

\begin{thm}\label{Elimit2}
Let $\tau_0>0$, $\nu\ge -\frac12$, $n\in\N_0$. Then we have
\[
\frac{\Ec_\nu^n({\rm i}(K'-\tau),\kk)}{\Ec_\nu^n({\rm i}(K'-\tau_0),\kk)}\to \frac{(\sinh\tau)^{1/2}\,Q_{n-\frac12}^{\nu+\frac12}(\cosh \tau)}{(\sinh\tau_0)^{1/2}\,Q_{n-\frac12}^{\nu+\frac12}(\cosh\tau_0)}
\]
as $k\to 0$ locally uniformly for $\operatorname{Re}\tau>0$.
The same results holds with $\Es_\nu^n$ in place of $\Ec_\nu^n$ for $n\in\N$.
\end{thm}
\begin{proof}
The function $w(\tau)=\Ec_\nu^n({\rm i}(K'-\tau),\kk)$, $0<\tau<2K'$, satisfies the differential equation
\begin{equation}\label{ode5}
 w''+\left(\nu(\nu+1)-a_\nu^n(\kk)-\nu(\nu+1)\ns^2(\tau,k')\right)w=0 .
\end{equation}
We would like to apply Lemma \ref{6:l1} to this differential equation but we cannot show that the coefficient of $w$ is negative if $-\frac12\le\nu<0$.
Instead of $w$ we consider the function
\[
v(\tau):=\sn^{-1/2}(\tau,k') w(\tau) .
\]
It satisfies the differential equation
\begin{equation}\label{odev}
 v''+ p(\tau,k)v'+q(\tau,k)v=0,
\end{equation}
where
\begin{gather*}
p(\tau,k):=\cn(\tau,k')\ds(\tau,k'),
\\
q(\tau,k):=\frac14k^2-a_\nu^n(\kk)-\frac34\dn^2(\tau,k')
-\bigg(\nu+\frac12\bigg)^2\cs^2(\tau,k') .
\end{gather*}
By Lemma \ref{4:l1},
$a_\nu^n(\kk)\ge -\tfrac14k^2$, so $q(\tau,k')\le-\frac14k^2<0$.
As $k\to0$,
\begin{gather*}
p(\tau,k)\to p(\tau):=(\cosh\tau\sinh\tau)^{-1},
 \\
q(\tau,k)\to q(\tau):=-n^2-\frac34\cosh^{-2}\tau -\bigg(\nu+\frac12\bigg)^2\sinh^{-2}\tau
\end{gather*}
locally uniformly for $\operatorname{Re}\tau>0$.
Up to a constant factor, the function
\[
u(\tau)=(\cosh\tau)^{1/2}\, Q_{n-\frac12}^{\nu+\frac12}(\cosh\tau)
\]
is the only solution of the differential equation
\[
u''+p(\tau)u'+q(\tau)u=0,
\]
which is bounded on the interval $[\tau_0, \infty)$.
Now we apply Lemma \ref{6:l1} with $a=\tau_0$ to the differential equation \eqref{odev} as $k\to 0$, and we obtain the statement of the theorem
with uniform convergence on $[\tau_0,b]$ for any $b>\tau_0$. The extension to locally uniform convergence for $\operatorname{Re}\tau>0$ follows from Lemma \ref{6:l2} applied to differential equation \eqref{ode5}. The proof for $\Es$ is very similar.
\end{proof}

Let $\F_\nu(s,h,\kk)$
be the solution of Lam\'e's equation belonging to the exponent $\nu+1$ at $z={\rm i}K'$, that is,
\begin{equation}\label{F}
 \F_\nu(s,h,\kk)=\tau^{\nu+1} \sum_{j=0}^\infty c_j \tau^j ,\qquad
 \tau=K'+{\rm i}s,\quad c_0=1.
\end{equation}
By analytic continuation, this function is well-defined in the strip $-K'<\operatorname{Im} s<K'$ using the principal branch of the power $(K'+{\rm i}s)^{\nu+1}$.
Note that $\F_\nu$ is real-valued for $s\in (-{\rm i}K',{\rm i}K')$.
Now we find the behavior of the function \eqref{F} with $h=a_\nu^n$ or $h=b_\nu^n$ as $k\to0$.

\begin{thm}\label{Flimit}
For $\nu\ge -\frac12$ and $n\in\N_0$, we have
\[
\tau^{-\nu-1}\F_\nu({\rm i}(K'-\tau),a_\nu^n(\kk),\kk)\to 2^{\nu+\frac12}\Gamma\bigg(\nu+\frac32\bigg) \tau^{-\nu-1}(\sinh\tau)^{1/2}P_{n-\frac12}^{-\nu-\frac12}(\cosh \tau)
\]
as $k\to 0$ locally uniformly for $|\operatorname{Im}\tau|<\pi$.
The same result holds with $b_\nu^n(\kk)$ in place of $a_\nu^n(\kk)$ for $n\in\N$.
\end{thm}
\begin{proof}
The function $w(\tau)=\F_\nu({\rm i}(K'-\tau),h,\kk)$ satisfies the differential equation
\[
\tau^2 w''=p(\tau,h,k)w,\qquad
p(\tau,h,k):=\tau^2\big(h-\nu(\nu+1)+\nu(\nu+1)\ns^2(\tau,k')\big).
\]
Now \eqref{k1} and the maximum principle for analytic functions give
\[
\tau\ns(\tau,k')\to \tau\coth\tau\qquad\text{as}\quad k\to 0\quad \text{locally uniformly for}\quad
|\operatorname{Im} \tau|<\pi.
\]
Using Theorem \ref{Elimit}, we see that
\begin{gather*}
p(\tau,a_\nu^n(\kk),k)\to p_\infty(\tau):=\tau^2\big(n^2-\nu(\nu+1)+\nu(\nu+1)\coth^2\tau\big)
\\ \phantom{p(\tau,a_\nu^n(\kk),k)\to p_\infty(\tau)\,}
{}= \tau^2\big(n^2+\nu(\nu+1)\sinh^{-2}\tau\big)
\end{gather*}
as $k\to 0$ uniformly for $|\operatorname{Im}\tau|<\pi$.
Note that the function
\[
w(\tau)=2^{\nu+\frac12}\Gamma\bigg(\nu+\frac32\bigg)(\sinh\tau)^{1/2} P_{n-\frac12}^{-\nu-\frac12}(\cosh \tau)
\]
solves the differential equation $\tau^2w''=p_\infty(\tau) w$ and it has the form
$w(\tau)=\tau^{\nu+1} v(\tau)$, where $v$ is analytic at $\tau=0$ with $v(0)=1$.
The latter equation follows from the representation of the associated
Legendre function of the first kind $P_\nu^\mu$
in terms of the hypergeometric function
\eqref{Pdefn}.
Therefore, the statement of the theorem follows from Lemma \ref{6:l3}.
The proof for $h=b_\nu^n$ is the same.
\end{proof}

\subsection[The expansion of the fundamental solution in flat-ring harmonics in the limit k to 0]
{The expansion of the fundamental solution in flat-ring harmonics \\ in the limit $\boldsymbol{k\to 0}$}

In Theorem \ref{5:main} we found the expansion of the fundamental solution in terms of internal and external flat-ring harmonics.
Let us write this expansion in the form
\begin{equation*}
\frac{1}{\|\r-\r^\ast\|}= \sum_{m\in\Z}^\infty\expe^{{\rm i}m(\phi-\phi^\ast)}\sum_{n=0}^\infty A_{m,n}(\tau,\tau^\ast,\psi,\psi^\ast,k),
\end{equation*}
where $s=2K-\psi$, $t=K'-\tau$, $\phi$ are flat-ring coordinates of $\r$,
$s^\ast=2K-\psi^\ast$, $t^\ast=K'-\tau^\ast$, $\phi^\ast$ are flat-ring coordinates of $\r^\ast$,
$t<t^\ast$, and
\begin{gather*}
A_{m,n}=\frac12 \bigg(\frac{\dn(\psi,k)-\cn(\psi,k)\dn(\tau,k')}{k'\sn(\tau,k')}\bigg)^{1/2}
\bigg(\frac{\dn(\psi^\ast,k)-\cn(\psi^\ast,k)\dn(\tau^\ast,k')}{k'\sn(\tau^\ast,k')}\bigg)^{1/2}
\\ \hphantom{A_{m,n}=}
{}\times \Ec_{|m|-\frac12}^n(2K-\psi,\kk)\Ec_{|m|-\frac12}^n(2K-\psi^\ast,\kk)
\\ \hphantom{A_{m,n}=}
{}\times \Ec_{|m|-\frac12}^n({\rm i}(K'-\tau),\kk)\Fc_{|m|-\frac12}^n({\rm i}(K'-\tau^\ast),\kk)+\cdots ,
\end{gather*}
where $\dots$ indicates a copy of the previous terms with $\Ec$, $\Fc$ replaced by $\Es$, $\Fs$ if $n\ge 1$ and represents $0$ if $n=0$.

The expansion \eqref{1:expansion} of the fundamental solution in toroidal harmonics
can be written in the form
\begin{equation*}
\frac{1}{\|\r-\r^\ast\|}=\sum_{m\in\Z}^\infty\expe^{{\rm i}m(\phi-\phi^\ast)}\sum_{n=0}^\infty B_{m,n}(\tau,\tau^\ast,\psi,\psi^\ast),
\end{equation*}
where
$\tau$, $\psi$, $\phi$ are toroidal coordinates of $\r$, $\tau^\ast$, $\psi^\ast$, $\phi^\ast$
are toroidal coordinates of $\r^\ast$ with $\tau>\tau^\ast$, and
\begin{gather*}
B_{m,n}=\frac1{\pi}(\cosh\tau-\cos\psi)^{1/2}(\cosh\tau^\ast-\cos\psi^\ast)^{1/2}
\epsilon_n \cos (n(\psi-\psi^\ast))
\\ \hphantom{B_{m,n}=}
{}\times(-1)^m\frac{\Gamma\big(n-m+\frac12\big)}{\Gamma\big(n+m+\frac12\big)}Q_{n-\frac12}^m(\cosh \tau)P_{n-\frac12}^m(\cosh \tau^\ast),
\end{gather*}
where
$\epsilon_n:=2-\delta_{n,0}$.
We now prove the main result of this section.
\begin{thm}
Let $m\in\Z$, $n\in \N_0$, $\tau,\tau^\ast>0$, $\psi,\psi^\ast\in\R$. Then we have
\begin{equation}\label{ABlimit}
 A_{m,n}(\tau,\tau^\ast,\psi,\psi^\ast,k)\to B_{m,n}(\tau,\tau^\ast,\psi,\psi^\ast)\qquad
 \text{as}\quad k\to0.
\end{equation}
\end{thm}
\begin{proof}
It is known \cite[formulas~(14.9.13) and~(14.9.14)]{NIST:DLMF} that
\begin{gather}
P_\nu^{m}(x)=\frac{\Gamma(\nu+m+1)}{\Gamma(\nu-m+1)} P_\nu^{-m}(x),\qquad
Q_\nu^m(x)= \frac{\Gamma(\nu+m+1)}{\Gamma(\nu-m+1)} Q_\nu^{-m}(x)\label{PQidentity}.
\end{gather}
Therefore, it is enough to consider $m\in\N_0$.
We have as $k\to 0$,
\begin{equation}\label{ABlimit1}
 \frac{\dn(\psi,k)-\cn(\psi,k)\dn(\tau,k')}{k'\sn(\tau,k')}\to
\frac{\cosh \tau-\cos\psi}{\sinh\tau}.
\end{equation}
It follows from Theorem \ref{Elimit} that
as $k\to0$,
\begin{gather}\label{ABlimit2}
\Ec_{m-\frac12}^n(2K-\psi,\kk)\Ec_{m-\frac12}^n(2K-\psi^\ast,\kk)
\to \frac2\pi \epsilon_n\cos \bigg(\!n\bigg(\frac\pi2-\psi\bigg)\!\bigg)
\cos \bigg(\!n\bigg(\frac\pi2-\psi^\ast\bigg)\!\bigg).
\end{gather}
If $n\in\N$, we have a similar result with $\Ec$ replaced by $\Es$ and $\cos$ replaced by $\sin$.
Set
\begin{gather*}
f(\tau,k):=\Ec^n_{m-\frac12}({\rm i}(K'-\tau),\kk),\qquad
g(\tau,k):=\Fc_{m-\frac12}^n({\rm i}(K'-\tau),\kk),
\\
f(\tau):=(\sinh \tau)^{1/2}Q_{n-\frac12}^m(\cosh \tau), \qquad
g(\tau):=(\sinh \tau)^{1/2}P^{-m}_{n-\frac12}(\cosh\tau).
\end{gather*}
Then we obtain from Theorems \ref{Elimit2} and~\ref{Flimit}, that as $k\to0$,
\begin{equation}\label{ABlimit3}
 \frac{f(\tau,k)g(\tau^\ast,k)}{W[g(\cdot,k),f(\cdot,k)]}
\to \frac{f(\tau)g(\tau^\ast)}{W[g,f]},
\end{equation}
where $W[g,f]$ denotes the Wronskian of $g$, $f$ (which is a constant).
Now, by definition, $W[g(\cdot,k),f(\cdot,k)]=1$ and, using \eqref{PQidentity}
and \cite[formula~(14.2.10)]{NIST:DLMF}
\[
W[P^\mu_\nu(x),Q_\nu^\mu(x)]=\expe^{{\rm i}\mu\pi} \frac{\Gamma(\nu+\mu+1)}{\Gamma(\nu-\mu+1)}\frac1{1-x^2} .
\]
Therefore, \eqref{ABlimit3} yields
as $k\to 0$,
\begin{gather}
\Ec^n_{m-\frac12}({\rm i}(K'-\tau),\kk)\Fc_{m-\frac12}^n({\rm i}(K'-\tau^\ast),\kk)
 \nonumber
 \\ \qquad
 {}\to(-1)^m \frac{\Gamma\big(n+m+\frac12\big)}{\Gamma\big(n-m+\frac12\big)}
\big(\sinh \tau\sinh \tau^\ast\big)^{1/2}Q_{n-\frac12}^m(\cosh \tau)
P^{m}_{n-\frac12}(\cosh \tau^\ast).
\label{ABlimit4}
\end{gather}
We have the same result when $\Ec$, $\Fc$ are replaced by $\Es$, $\Fs$, respectively.
If we combine \eqref{ABlimit1}, \eqref{ABlimit2} and~\eqref{ABlimit4},
and noting that
\[
\cos \bigg(\!n\bigg(\frac\pi2\!-\psi\bigg)\!\bigg)\cos \bigg(\!n\bigg(\frac\pi2\!-\psi^\ast\bigg)\!\bigg)\!+\sin \bigg(\!n\bigg(\frac\pi2\!-\psi\bigg)\!\bigg)\sin \bigg(\!n\bigg(\frac\pi2\!-\psi^\ast\bigg)\!\bigg) \!=\cos(n(\psi\!-\psi^\ast)),
\]
we obtain \eqref{ABlimit}.
\end{proof}

\subsection*{Acknowledgements}

We greatly appreciate the comments of the referees which led to improvements of the paper.

\pdfbookmark[1]{References}{ref}
\LastPageEnding

\end{document}